\documentclass[12pt]{amsart}
\usepackage{amssymb, amsmath}
\usepackage{amsfonts,amsbsy,amssymb,amsthm}
\usepackage[usenames,dvipsnames]{color}
\usepackage{url}

\usepackage{amsmath}
\usepackage{mathtools}
\usepackage{amsthm}
\usepackage{amssymb}
\usepackage{amsfonts}
\usepackage{graphics}
\usepackage{graphicx}
\usepackage{subcaption}
\usepackage{float}
\usepackage{multicol}
\usepackage{hyperref}
\usepackage{todonotes}

\newtheorem{theorem}{\bf Theorem}[section]

\newtheorem{corollary}[theorem]{\bf Corollary}

\newtheorem{lemma}[theorem]{\bf Lemma}

\newtheorem{thm}[theorem]{\bf Theorem}
\newtheorem{cor}[theorem]{\bf Corollary}

\theoremstyle{definition}
\newtheorem{definition}[theorem]{\bf Definition}
\newtheorem{remark}[theorem]{\bf Remark}
\newtheorem{example}[theorem]{\bf Example}
\newtheorem{question}[theorem]{\bf Question}

\newcommand{\ba}{\begin{array}}
\newcommand{\ea}{\end{array}}

\def \R{{\mathbb R}}
\def \C{{\mathbb C}}

\def \K{{\mathbb K}}

\def \I{{\mathcal I}}

\def \x{{\mathbf{x}}}
\def \w{{\mathbf{w}}}

\def \e{{\mathbf{e}}}

\def \P{{\mathbb P}}

\def \Delta{\triangle}

\def \x{{\mathbf{x}}}
\def \y{{\mathbf{y}}}
\def \z{{\mathbf{z}}}

\def \w{{\mathbf{w}}}

\def \aaa{{\mathbf{a}}}

\DeclareMathOperator{\inter}{int} % interior
\DeclareMathOperator{\GL}{GL} % general linear group
\DeclareMathOperator{\init}{in}

\DeclareMathOperator{\Res}{Res} % Resultant
\DeclareMathOperator{\Disc}{Disc}
\DeclareMathOperator{\vol}{vol}
\DeclareMathOperator{\New}{New}

\newcommand{\compl}{\mathsf{c}}

\usepackage{multicol}
\usepackage{microtype}

\title[Imaginary projections: Complex versus real coefficients]
{Imaginary projections: \\ Complex versus real coefficients}

\author{Stephan Gardoll}

\author{Mahsa Sayyary Namin}

\author{Thorsten Theobald}

\address{Goethe-Universit\"at, FB 12 -- Institut f\"ur Mathematik,
  Postfach 11 19 32, 60054 Frankfurt am Main, Germany}

\email{\{gardoll,sayyary,theobald\}@math.uni-frankfurt.de}

\thanks{This work was supported through DFG grant TH 1333/7-1.} 

\begin{document}
	
\begin{abstract}
	Given a multivariate complex polynomial ${p\in\mathbb{C}[z_1,\ldots,z_n]}$, the imaginary projection $\mathcal{I}(p)$ of $p$ is defined as the projection of the variety $\mathcal{V}(p)$ onto its imaginary part. We focus on studying the imaginary projection of complex polynomials and we state explicit results for certain families of them with arbitrarily large degree or dimension. Then, we restrict to complex conic sections and give a full characterization of their imaginary projections, which generalizes a classification for the case of real conics. That is, given a bivariate complex polynomial $p\in\C[z_1,z_2]$ of total degree two, we describe the number and the boundedness of the components in the complement of $\mathcal{I}(p)$ as well as their boundary curves and the spectrahedral structure of the components. We further show a realizability result for strictly convex complement components which is in sharp contrast to the case of real polynomials.
\end{abstract}

\maketitle
\section{Introduction\label{se:intro}}

Given a polynomial $p\in\C[\z]:=\C[z_1, \ldots ,z_n]$, the imaginary projection $\I(p)$ as introduced in \cite{jtw-2019} is the projection of the variety $\mathcal{V}(p)\subseteq\C^n$ onto its imaginary part, that is,
\begin{equation}
\label{eq:imagproj1}
\ \I(p)= \ \left\{ \z_{\rm im} = ((z_1)_{\rm im}, \ldots, (z_n)_{\rm im}) \ : \ \mathbf{z} \in \mathcal{V}(p)\right\} \ \subseteq \ \R^n,
\end{equation}
where $(\cdot)_{\rm im}$ is the imaginary part of a complex number.
Recently, there has been wide-spread research interest in mathematical branches which are directly connected to the imaginary projection of polynomials. 

As a primary motivation, the imaginary projection provides a 
comprehensive geometric view for notions of \emph{stability of 
	polynomials} and generalizations thereof. A polynomial $p\in\C[\z]$ is called \textit{stable}, if $p(\z)=0$ implies $(z_j)_{\rm im}\leq 0$ for some $j\in[n]$.
In terms of the imaginary projection $\I(p)$, we can express the stability of $p$ as
the condition $\I(p)\cap\R^n_{>0}=\emptyset$. Stable polynomials
have applications in many branches of mathematics
including combinatorics (\cite{braenden-hpp} and see \cite{brown-wagner-2020}
for the connection of the imaginary projection to combinatorics), 
differential equations \cite{borcea-braenden-2010},
optimization \cite{straszak-vishnoi-2017},
probability theory \cite{bbl-2009}, and
applied algebraic geometry \cite{volcic-2019}.
Further application areas include
theoretical computer science \cite{mss-interlacing1, mss-interlacing2},
statistical physics \cite{borcea-braenden-leeyang1}, and control theory \cite{ms-2000}, see also the surveys \cite{pemantle-2012} and
\cite{wagner-2010}. 

Recently, various generalizations and variations of the stability
notion have been studied, such as stability with respect to
a polyball \cite{gkv-2016,gkv-2017},
conic stability \cite{dgt-conic-pos-map-2019,joergens-theobald-conic},
Lorentzian polynomials \cite{braenden-huh-2020}, or positively
hyperbolic varieties \cite{rvy-2021}. Exemplarily, regarding the 
conic stability,
a polynomial $p\in\C[\z]$ is called \textit{$K$-stable} for a proper cone $K\subset \R^n$ if $p(\z)\neq 0$, whenever $\z_{\rm im}\in \inter K$,
where $\inter$ is the interior. In terms of
the imaginary projection, this condition can be equivalently expressed as 
$\I(p)\cap \inter K=\emptyset$. 

Another motivation comes from the close connection of the imaginary
projection to 
hyperbolic polynomials and hyperbolicity 
cones \cite{garding-59}.
As shown in \cite{joergens-theobald-hyperbolicity}, in case of a 
real \emph{homogeneous} polynomial $p$, 
the components of the complement 
$\I(p)^\compl$ coincide with the hyperbolicity cones
of $p$. 
These concepts play a central role in hyperbolic programming,
see \cite{gueler-97,naldi-plaumann-2018,
	nesterov-tuncel-2016,saunderson-2019}.
A prominent open question in this research direction is the generalized
Lax conjecture, which claims that every hyperbolicity cone is 
spectrahedral, see \cite{vinnikov-2012}.
Representing
convex sets by spectrahedra is not only motivated by the general 
Lax conjecture, but also by the question of effective handling convex
semialgebraic sets (see, for example, \cite{bpt-2013,kpv-2015}).
Recently, the conjecture that every convex semialgebraic set would be 
the linear projection of a spectrahedron, the 
``Helton-Nie conjecture'', has 
been disproven by Scheiderer \cite{scheiderer-spectrahedral-shadows}.

Moreover, the imaginary projection closely relates to and complements 
the notions of \emph{amoebas}, as introduced by Gel'fand, Kapranov 
and Zelevinsky \cite{gkz-1994}, and \emph{coamoebas}. 
The amoeba $\mathcal{A}(p)$
of a polynomial $p$
is defined as 
{\small	$\,{\!\mathcal{A}(p)\! := \!\{(\ln|z_1|, \ldots,\ln|z_n|) \!:\! 
		\mathbf{z} \in\! \mathcal{V}(p) \cap (\C^*)^n \}}$},
so it considers the logarithm of the absolute value of a complex
number rather than its imaginary part. The coamoeba of a polynomial
deals with the phase of a complex number. Each of these three viewpoints
of a complex variety gives a set in a real space with the characteristic
property that the complement of the closure consists of finitely
many \emph{convex} connected components. See 
\cite{forsgard-johansson-2015}, \cite{gkz-1994} 
and  \cite{jtw-2019}
for the convexity properties of amoebas, coamoebas, 
and imaginary projections, respectively. Due to their convexity
phenomenon, these structures provide natural classes in recent
developments of convex algebraic geometry. 

For amoebas, an exact upper bound on the number of components in 
the complement is known \cite{gkz-1994}.
For the coamoeba of a polynomial $p$, it has been conjectured that
there are at most $n!\vol\New(p)$ connected components in the complement,
where $\vol$ denotes the volume and $\New(p)$ the Newton polytope
of $p$, see \cite{forsgard-johansson-2015}
for more background as well as a proof for the special case
$n=2$.
For imaginary projections, a tight upper bound is known in the homogeneous
case \cite{joergens-theobald-hyperbolicity}, but for the non-homogeneous
case there only exists a lower bound \cite{jtw-2019}.

Currently, no efficient method is known to calculate the imaginary 
projection for a general real or complex polynomial. For
some families of polynomials, the imaginary projection has been
explicitly characterized,
including complex linear polynomials and real quadratic polynomials,
see \cite{jtw-2019} and \cite[Proposition 3.2]{joergens-theobald-conic}. 
However, since imaginary projections for non-linear
complex polynomials exhibit new structural phenomena compared
to the real case,
even the characterization of the imaginary projection of complex
conics had remained elusive so far.

Our primary goal is to reveal fundamental and surprising differences between imaginary projections of \emph{real polynomials} and \emph{complex polynomials}.
In fixed degree and dimension, for a polynomial $p$ with non-real coefficients, the algebraic degree of the {\it boundary} of the imaginary projection  $\partial\I(p):= \overline{\I(p)}\cap\overline{\I(p)^\compl}$
 can be higher than the case of real coefficients. Here $(.)^\compl$ and $\overline{(.)}$ are the complement and Euclidean closure, respectively.
These incidences already begin when the degree and dimension are both two. However, the contrast is not only concerning the boundary degrees, but also the arrangements and the strict convexity of the components in $\I(p)^\compl$.

We start with structural results which serve to work out the differences 
between the case of real and complex coefficients. Our first result is a sufficient
criterion on the roots of the \emph{initial form} 
of an arbitrarily large degree non-real 
bivariate complex polynomial to have the real plane as its imaginary projection,
see Theorem~\ref{th:EvenDeg} and Corollary~\ref{co:wholeplane2}.

Next, we characterize the imaginary projections of $n$-dimensional
multivariate complex quadratics with hyperbolic initial form,
see Theorem~\ref{th:ndimhyperbol1} and Corollary~\ref{co:ndimhyperbol2}. 

In the two-dimensional case,
although by generalizing from real to complex conics, the bounds on the number of bounded and unbounded components in the complement of the imaginary projections remain unchanged,
the
possible arrangements of these components, strictness of their convexity, and the algebraic degrees of their boundaries strongly differ. See Corollaries \ref{co:alg-degrees} and
\ref{cor:oneUnbdd}.
For conic sections with real coefficients, it was shown 
by J\"orgens, Theobald, and de~Wolff \cite{jtw-2019} that the boundary   
$\partial\I(p)$ consists of pieces which are algebraic
curves of degree at most two. In sharp contrast to this,
for complex polynomials, the boundary may not be algebraic and the degree of its irreducible pieces can go up to 8.
For example, despite the simple expression of the polynomial $p = z_1^2+{\rm i}z_2^2+z_2$, an exact description 
of 
$\I(p)$ is
\begin{equation}
\label{eq:example1}
\begin{array}{r@{\hspace*{0.5ex}}l}
\mathcal{I}(p) \ = \ & \{ y \in \R^2 \, : \,
-64 y_1^8-128 y_1^4 y_2^4-64 y_2^8
+256 y_1^4 y_2^3+256 y_2^7-272 y_1^4 y_2^2
\\ [1ex]
& \, -400 y_2^6+144 y_1^4 y_2+304 y_2^5-27 y_1^4-112 y_2^4+16 y_2^3 \le 0\}
\setminus \{(0,1/2)\},
\end{array}
\end{equation}
and the describing polynomial in~\eqref{eq:example1} is irreducible over $\C$.
In this example, the set ${\I(p)}^{\compl}$ consists of a single convex connected and bounded component.
Any polynomial vanishing on the boundary will also vanish on the single
point $(0,1/2)$ which is not part of the boundary $\partial \I(p)$.
Thus, $\partial \I(p)$ is not algebraic.
See Figure~\ref{fi:example1} for an illustration and we return to this
example in Section \ref{se:higherdegree} and at the end of Section~\ref{se:non-hyperbolic}.

\begin{figure}[ht]
	\begin{subfigure}{.40\textwidth}
		\centering
		\includegraphics[width=1.6in]{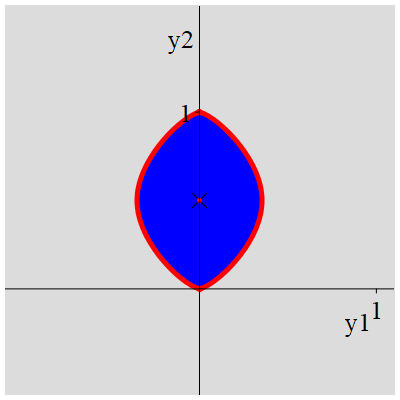}
		\caption{}
	\end{subfigure}%
	\begin{subfigure}{.40\textwidth}
		\centering
		\includegraphics[width=1.6in]{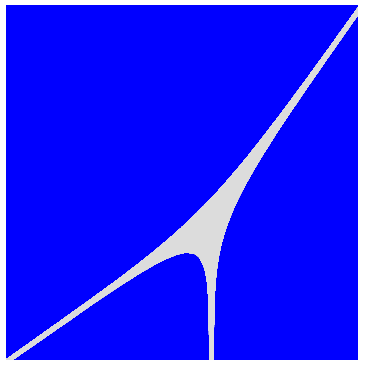}
		\caption{}
	\end{subfigure}%
	\caption{{\small(A)} The gray area and its boundary form the imaginary projection $\I(p)$ of ${p = z_1^2+{\rm i}z_2^2+z_2}$. The polynomial in \eqref{eq:example1} vanishes on the red curve, which consists of a single point and another bounded component. The complement $\I(p)^\compl$ contains the single point and it is
		bounded by the other component. {\small(B)} The amoeba of $p$ is shown in gray.}
	\label{fi:example1}
\end{figure}
Since the topology of the imaginary projection in $\R^n$ is invariant under the action of $G_n:=\C^n\rtimes \GL_n(\R)$, that is the semi-direct product of $\GL_n(\R)$ and complex translations,
the problem to understand the imaginary projections
naturally leads to a polynomial classification problem. 

\medskip

As starting point, recall that 
under the action of the affine group
$\text{Aff}(\C^2)$, there are precisely five orbits for complex conics, with the following representatives:			
\[
\begin{matrix}
z_1^2 \text{ (one line)},&&&
z_1^2+1 \text{ (two parallel lines)},&&&
z_1^2-z_2 \text{ (parabola)},
\end{matrix}
\]\vspace{-6mm}\[
\begin{matrix}
z_1^2+z_2^2 \text{ (two crossing lines)},&&&
z_1^2+z_2^2-1 \text{ (circle)}.
\end{matrix}
\]

However, the arrangement of the components in $\I(p)^\compl$
is not invariant under the action of $\text{Aff}(\C^2)$, but only under its restriction to $G_2$. There are several other related classifications of complex conic sections. Newstead \cite{Newstead} has classified the set of projective complex conics under real linear transformations. However, out of a projective setting his method becomes ineffective as it is based on the arrangements of four intersection points between a conic and its conjugate. 
On the other hand, by considering the real part and the
imaginary part of a complex conic $p$, under
the action of $G_2$ the classification of conic sections has some relations
to the problem of
classifying pairs of real conics. 
Systematic
classifications of this kind
are mostly done in the projective setting and are well understood. See 
\cite{briand-2007,levy-1964,petitjean-2010,uhlig-1976}. However, 
those classifications rely
on the invariance of the number and multiplicity of real intersection points between the two real conics. 
The drawback here is that under complex translations on $p$, these numbers are not invariant anymore, except at infinity.

To capture the invariance under $G_2$, we
develop a novel classification based on the initial forms of complex conics. This classification is adapted to the imaginary projection 
and it is rather fine but coarse enough to allow handling the inherent
algebraic degree of 8 in the boundary description of the imaginary projection.

Finally, we show that non-real complex conics can significantly improve a realization result on the complement of the imaginary projections. In \cite{joergens-theobald-hyperbolicity}, for any given integer $k\ge 1$, they present a polynomial $p$ of degree $d=4\lceil \frac{k}{4}\rceil+2$ as a product of real conics, such that $\I(p)^\compl$ has at least $k$ components that are strictly convex and bounded.  
Using non-real conics, we furnish a degree $d/2+1$ polynomial having exactly $k$ components with these properties. See Theorem~\ref{th:StrictlyConvexComplex} and Question \ref{ques:deg}.

The paper is structured as follows. Section \ref{se:prelim} provides our notation and the necessary background on the imaginary projection of polynomials and contains
the classification of the imaginary projection for the case of real conics. Section~\ref{se:higherdegree} 
deals with complex plane curves and provides a highlighting example where the complex versus real coefficients make a remarkable difference in the complexity of the imaginary projection. 
Moreover, we determine a family of arbitrarily large degree non-real plane curves with a full-space imaginary projection, based on the arrangements of roots of the initial form. 
In Section~\ref{se:QuadraticsWithHyperbolicInit}, we set the degree to be two and let the dimension grow and we classify the imaginary projections of complex quadratics with hyperbolic initial form.
In Sections~\ref{se:mainclassification} 
and~\ref{se:non-hyperbolic}, we restrict the degree and dimension both to be two and we provide a full classification of the imaginary projections for affine complex conics based on their initial forms. Moreover, we determine in which classes the components in the complement of the imaginary projection have a spectrahedral description and also state them explicitly.

Section~\ref{se:mainclassification} 
contains our main classification theorems and the corollaries differentiating the cases of complex and real coefficients. The part where the initial form is hyperbolic is already covered in \ref{se:QuadraticsWithHyperbolicInit}.  Each subsection of  Section~\ref{se:non-hyperbolic} treats one of the remaining classes and explains their spectrahedral structure. In particular, we show that the only class where the components in the complement are not necessarily spectrahedral is the case where the initial form has two distinct non-real roots in $\P^1_\C$ such that they do not form a complex conjugate pair. In Section~\ref{se:convex}, we
prove a realization result for strictly convex complement components, which highlights another contrast between the imaginary projections of complex and real polynomials.
Section~\ref{se:outlook} gives some open questions.

	%%%%%%%%%%%%%%%%%%%%%%%%%%%%%%%%%%%
	%%%%%%%%%%%%%%%%%%%%%%%%%%%%%%%%%%%
	%%%%%%%%%%%%%%%%%%%%%%%%%%%%%%%%%%%
	%%%%%%%%%%%%%%%%%%%%%%%%%%%%%%%%%%%
	
	\section{Preliminaries and background}\label{se:prelim}
	
    For a set $S\subseteq\R^n$, we denote by 
	$\overline{S}$ the topological closure of $S$ with respect to the Euclidean topology on $\R^n$ and by $S^{\compl}$
	the complement of $S$ in $\R^n$. The \textit{algebraic degree} of $S$ is the degree of its closure with respect to the 
Zariski topology. The set of non-negative and the set of strictly positive real numbers are 
abbreviated by $\R_{\ge 0}$ and $\R_{>0}$ throughout the text.
	Moreover,
	bold letters will denote $n$-dimensional vectors.
   By $\P^n$ and $\P^n_{\R}$, we denote the $n$-dimensional complex and real
   projective spaces, respectively.

\medskip
	
	For a polynomial $p \in \C[\mathbf{z}]$, the imaginary
	projection $\mathcal{I}(p)$ is defined in~\eqref{eq:imagproj1} and its boundary $ \overline{\I(p)}\cap\overline{\I(p)^\compl}$
is denote by $\partial\I(p)$.
\begin{theorem}\cite{jtw-2019}
Let $p\in\C[\z]$ be a complex polynomial. The set $\overline{\I(p)}^\compl$ consists 
of a finite number of convex connected components. 
\end{theorem}	
	
We denote by $a_{\rm{re}}$ and $a_{\rm{im}}$ the real and the imaginary parts of a complex number $a\in\C$, i.e., $a$ is written in the form $a_{\rm re}+{\rm i}a_{\rm im}$, such that $a_{\rm re},a_{\rm im}\in\R$.  Let $p\in\C[\z]$ be a complex polynomial. After substituting 
$z_j = x_j+{\rm i}y_j$
for all $1\le j\le n$, the complex polynomial can be written in the form 
\[p(\z) =p_{\rm re}(\x,\y)+{\rm i}p_{\rm im}(\x,\y),\]
 such that $p_{\rm re},p_{\rm im}\in\R[\x,\y]$. We call the real polynomials $p_{\rm re}$ and $p_{\rm im}$, the \textit{real part} and the \textit{imaginary part} of $p$, respectively. Thus, finding $\I(p)$ is equivalent to determining the values of $\y$ for which  the real polynomial system
 \begin{equation}\label{PolySystem} 
p_{\rm re}\,(\x,\y)=0 \; \text{ and } \;
p_{\rm im}(\x,\y)=0
 \end{equation}
has real solutions for $\x$.

	\begin{definition}\label{def:complexConic}
		Let $p\in\C[z_1,z_2]$ be a quadratic polynomial, i.e., $p = a z_1^2 + b z_1 z_2 + c z_2^2 + d z_1 + e z_2 +f$ such that $a,b,c,d,e,f\in\C$. We say that $p$ is the defining polynomial of a complex conic, or shortly, a \textit{complex conic} if its total degree equals two, i.e., at least one of the coefficients $a,b$, or $c$ is non-zero.  A complex conic 
$p$ is called a \textit{real conic} if all coefficients of $p$ are real.
	\end{definition}
		
	The following lemma from \cite{jtw-2019} shows how real linear transformations and complex translations act on the imaginary projection. These are the key ingredients for computing the imaginary projection of every class of conic sections. 
	\begin{lemma}\label{le:group-actions-improj}
		Let $p\in\C[\z]$ and $A\in\R^{n\times n}$ be an invertible matrix. Then \[{\I(p(A\z))=A^{-1}\I(p(\z)).}\]
	
		Moreover, a real translation $\z\mapsto \z+\aaa, \ \aaa\in\R^n$ does not change the imaginary projection. An imaginary translation $\z\mapsto \z+{\rm i}\aaa, \ \aaa\in\R^n$ shifts the imaginary projection into the direction $-\aaa$.
	\end{lemma}
By the previous lemma, to classify the imaginary projection of polynomials we consider their orbits under the action of the group $G_n:=\C^n\rtimes \text{GL}_n(\R)$, given by real linear transformations and complex translations. Further let 
 $\text{Aff}(\K^n):=\K^n\rtimes\text{GL}_n(\K)$ be the general affine group for $\K=\R$ or $\K=\C$. The real dimensions of these groups are
\[
\begin{matrix}
\dim_\R(\text{Aff}(\C^n))=2\dim_\R(\text{Aff}(\R^n))=2(n^2+n),&&\dim_\R(G_n)=n^2+2n.
\end{matrix}
\]
 
Up to the action of $G_2$, a real conic $p\in\R[z_1,z_2]$ is equivalent to a conic given by one of the following polynomials.
 			\setlength{\columnsep}{0pt}
\vspace{-2mm}	\begin{multicols}{2}
	\begin{itemize}
		\item[($i$)] $z_1^2+z_2^2-1$ (ellipse),
		\item[($ii$)] $z_1^2-z_2^2-1$ (hyperbola),
		\item[($iii$)] $z_1^2+z_2$ (parabola),
		\item [($iv$)]$z_1^2+z_2^2+1$ (empty set),
		\item[($v$)] $z_1^2-z_2^2$ (pair of crossing lines),
	\item[($vi$)] $z_1^2-1$ (parallel lines/one line $z_1^2$),
		\item [($vii$)]$z_1^2+z_2^2$ (isolated point),
		\item [($viii$)]$z_1^2+1$ (empty set).
	\end{itemize}
\end{multicols}

In \cite{jtw-2019}, a full classification of the imaginary projection for real quadratics was shown. In particular, the following theorem is the classification for real conics. For illustrations of the cases, see 
Figure~\ref{fi:real-classification}. The theorem that comes after provides the imaginary projection of some families of real quadratics. Furthermore, they state the subsequent question as an open problem. 

\begin{theorem}\label{th:RealConicChar}
	Let $p\in\R[z_1,z_2]$ be a real conic. For the normal forms 
	(i)--(viii) from above, the imaginary projections $\I(p)\subseteq\R^2$ 
	are as follows.
	\hspace{-3mm}		\begin{multicols}{2}
		\begin{itemize}
			\item[($i$)] $\I(p) =\R^2$,
			\item[($ii$)] $\I(p) = \{-1\le y_1^2-y_2^2<0\}\cup\{\mathbf{0}\}$,
			\item[($iii$)] $\I(p) =\R^2\setminus\{(0,y_2):y_2\neq 0\}$, 
			\item [($iv$)]$\I(p) =\{\y\in\R^2:y_1^2+y_2^2-1\ge 0\}$, 
			\item[($v$)] $\I(p) =\{\y\in\R^2:y_1^2=y_2^2\}$, 
			\item[($vi$)] $\I(p) =\{\y\in\R^2:y_1=0\}$, 
			\item [($vii$)]$\I(p) =\R^2$,
			\item [($viii$)]$\I(p) =\{\y\in\R^2:y_1=\pm 1\}$.
		\end{itemize}
	\end{multicols}
\end{theorem}	

\begin{theorem}\label{th:real-Quad}
	Let $p\in\C[z_1,\dots,z_n]$ be $p = \sum_{i=1}^{n-1}z_i^2-z_{n}^2+k$ for  $k\in\{\pm 1\}$. Then 
	\[
	\I(p) = \begin{cases}
	\left\{\y \in \R^n \ : \ y_n^2<\sum_{i=1}^{n-1} y_i^2 \right\}
	\cup \{\mathbf{0}\}  & \text{if } k =1,\vspace{1mm} \\
	\vspace{1mm}
	\left\{\y \in \R^n \ : \ y_{n}^{2}-\sum_{i=1}^{n-1} y_i^2 \le 1 \right\}
	& \text{if } k =-1. \\
	\end{cases}
	\]
\end{theorem}
\begin{figure}
	\begin{subfigure}{.24\textwidth}
		\centering
		\includegraphics[width=1.2in]{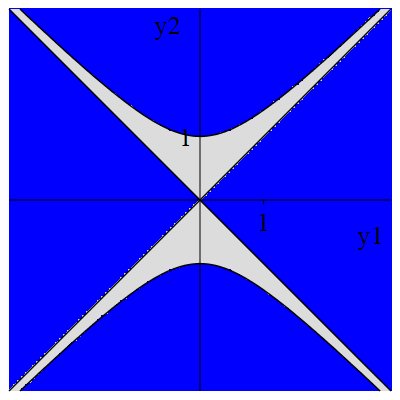}
	\end{subfigure}%
	\begin{subfigure}{.24\textwidth}
		\centering
		\includegraphics[width=1.2in]{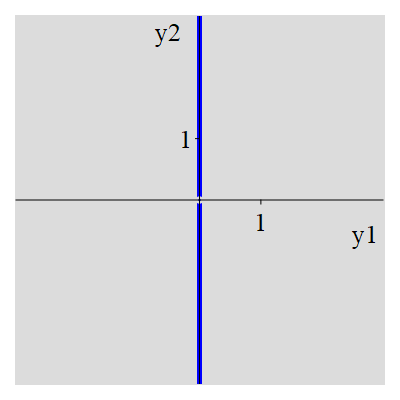}
	\end{subfigure}%
	\begin{subfigure}{.24\textwidth}
		\centering
		\includegraphics[width=1.2in]{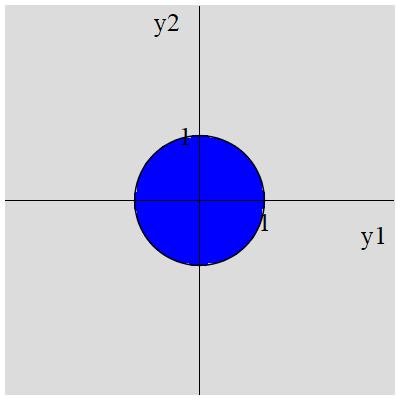}
	\end{subfigure}
	\\
	\begin{subfigure}{.24\textwidth}
		\centering
		\includegraphics[width=1.2in]{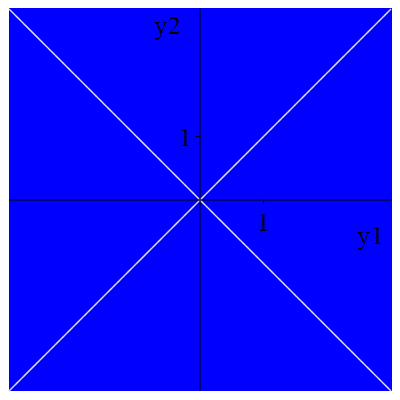}
	\end{subfigure}%
	\begin{subfigure}{.24\textwidth}
		\centering
		\includegraphics[width=1.2in]{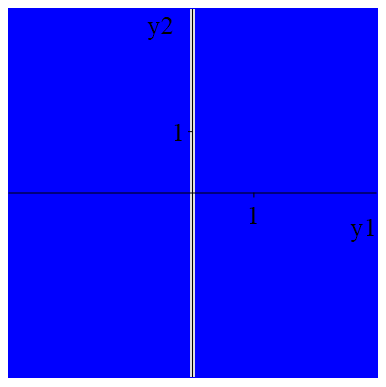}
	\end{subfigure}%
	\begin{subfigure}{.24\textwidth}
		\centering
		\includegraphics[width=1.2in]{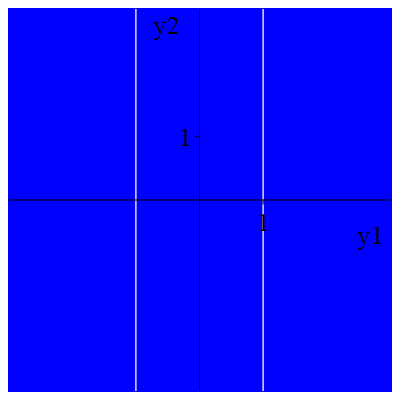}
	\end{subfigure}
	\caption{The imaginary projections of the real conic sections and their complements are colored in gray and blue, respectively. The cases $(i)$ and $(vii)$ are skipped, as their
		imaginary projection is the whole plane.}\vspace{-2mm}
	\label{fi:real-classification}
\end{figure}
The following question, which is true for real quadratics $p\in\C[\z]$,
 was asked in \cite[Open problem 3.4]{jtw-2019}. 
In Section~\ref{subs:real-non-real},
we show that it is not true in general even for complex conics.

\begin{question}\label{que:open-close}
	Let $p\in\C[\z]$ be a polynomial. Is $\I(p)$ open if and only if $\I(p)=\R^n$?
\end{question} 

We use the {\it initial form} of $p$ abbreviated by $\init(p)(\z)=p^h(\z,0)$ , where $p^h$ is the homogenization of $p$.  The initial form consists of the terms of $p$ with the maximal total degree. Furthermore, a complex polynomial $p \in \C[\z]$
is called \emph{hyperbolic} w.r.t. $\e\in\R^n$ if the univariate polynomial $t\mapsto p(\x+t\e)$ is real-rooted. Note that any hyperbolic polynomial is a, possibly complex, multiple of a real polynomial.

\medskip

Finally, a \textit{spectrahedron} is a set of the form
\vspace{-2mm}
\[
\{ \x \in \R^n \ : \ A_0 + \sum_{j=1}^n A_j x_j \succeq 0\},\vspace{-2mm}
\]
where $A_1, \ldots, A_n$ are real symmetric matrices of size $d$. Here, ``$\succeq 0$'' denotes the positive semidefiniteness
of a matrix. We also speak of a spectrahedral set if the set is given by positive
definite conditions, i.e., by strict conditions.

\medskip

	%%%%%%%%%%%%%%%%%%%%%%%%%%%%%%%%%%
	%%%%%%%%%%%%%    Higher Degree  %%%%%%%%%%%%
	%%%%%%%%%%%%%%%%%%%%%%%%%%%%%%%%%%
	%%%%%%%%%%%%%%%%%%%%%%%%%%%%%%%%%%
	
	\section{Imaginary projections of complex plane curves\label{se:higherdegree}}
 In this section, we determine the imaginary projection of some families of arbitrarily high degree complex plane curves.
Our point of departure is the characterization of real conics in 
Theorem \ref{th:RealConicChar}. In the following example, which is an affine version of case ($B_{+}$) in Newstead's classification \cite{Newstead}, 
we show that by allowing non-real coefficients the imaginary projection of a complex conic can significantly change in terms of the algebraic degree of its boundary. See Corollary \ref{co:alg-degrees}. 
 
 \begin{remark}\label{re:quarticRoots}
 	Recall that the discriminant of a univariate polynomial
 	$p(z) = \sum_{j=0}^n a_j z^j$ is given
 	by $\Disc(p) = (-1)^{\frac{1}{2}n(n-1)}\frac{1}{a_n} \Res(p,p')$,
 	where $\Res$ denotes the resultant. For a quartic, having negative discriminant implies the existence of a real root. However, a positive discriminant can correspond to either four 
 	real roots or none. Let 	
 	{\small
 		\[
 		P  =  8 a_2 a_4 - 3 a_3^2, \,
 		R  =  a_3^{3}+8a_1a_4^{2}-4a_4a_3a_2,\,
 		D  =  64a_4^{3}a_0-16a_4^{2}a_2^{2}+16a_4a_3^{2}a_2-16a_4^{2}a_3a_1-3a_3^{4}.
 		\]}
 	If $\Disc(p)>0$, then $p=0$ has four real roots if $P < 0$ and $D<0$, and no real roots otherwise. Finally, if the discriminant is zero, the only conditions under which there is no real solution is having $D=R=0$ and $P>0$
 	(see, e.g., \cite[Theorem 9.13 (vii)]{janson-2011}).
 \end{remark} 
 
 \begin{example}\label{ex:caseB}
 	Let $p=z_1^2+{\rm i}z_2^2+z_2$.  For simplifying the calculations, we use the translation $z_2\mapsto z_2+{\rm{i}}/2$ to eliminate the linear term.
 	This turns the equation $p=0$ into 
 	$
 	q := z_1^2+{\rm i}z_2^2+{\rm{i}}/4=0.
 	$
 	Building the real polynomial system as introduced in (\ref{PolySystem}) implies
 	\[q_{\mathrm{re}} = x_1^2-2x_2y_2-y_1^2 = 0 \; \text{ and } \;
 	q_{\mathrm{im}} = 4x_2^2 +8x_1y_1 - 4y_2^2 + 1= 0.
 	\]
 	
 	First assume $y_1\neq 0$. Substituting $x_1$ from $q_{\mathrm{im}}=0$ into $q_{\mathrm{re}}=0$ gives\[
 	16x_2^4 + (-32y_2^2 + 8)x_2^2 - 128y_1^2y_2x_2 - 64y_1^4 + 16y_2^4 - 8y_2^2 + 1 = 0.
 	\]
 	
 	We calculate the discriminant of the above equation with respect to $x_2$. By the previous remark, there  is a real solution for $x_2$ if the discriminant is negative, i.e.,
 	\[
 	-64y_1^8 - 128y_1^4y_2^4 - 64y_2^8 - 80y_1^4y_2^2 + 48y_2^6 + y_1^4 - 12y_2^4 + y_2^2< 0.
 	\]
 	
 	Now we need to check the conditions where the discriminant is zero or positive. To show the positive discriminant implies no real solution for $x_2$,
we rewrite the condition with the substitution $u = y_1^4$:
 	\[
 	\Delta:=-64u^2 + (-128y_2^4 - 80y_2^2 + 1)u - 64y_2^8 + 48y_2^6 - 12y_2^4 + y_2^2>0.
 	\]
 	
 	It is a quadratic polynomial in $u$ with negative leading coefficient. It can only be positive between the two roots for $u$ in $\Delta=0$. Those are
 	\[
 	-y_2^4 - \frac{5}{8}y_2^2 + \frac{1}{128} \pm \frac{\sqrt{32768y_2^6 + 3072y_2^4 + 96y_2^2 + 1}}{128}.
 	\]
 	
 	To obtain $\Delta >0$, we need to have a solution $u>0$, i.e., we need to have either $-y_2^4 - \frac{5}{8}y_2^2 + \frac{1}{128}\ge 0$ or otherwise 
 	{\small
 		\[
 		\left(-y_2^4 - \frac{5}{8}y_2^2 + \frac{1}{128}\right)^2>\frac{32768y_2^6 + 3072y_2^4 + 96y_2^2 + 1}{128^2}.
 		\] }
 	
 	The first inequality implies $y_2^2\le \frac{3\sqrt{3}-5}{16}$ and after simplifications the second inequality implies $y_2^2< 1/4$. The polynomial $P$ from the previous remark for the quartic polynomials evaluates to $
 	4(1-4y_2^2),$ which is positive for $y_2^2< 1/4$. Therefore, for $\Delta>0$, there is no real solution for $x_2$.  It remains now to consider the case $\Delta=0$. Since $y_1\neq 0$, to have $R=-262144y_2y_1^2=0$ we need $y_2=0$. Substituting $y_2=0$ in $D=0$ implies $-4096y_1^4 - 960=0$, which is a contradiction. Therefore, if $y_1\neq 0$, the imaginary projection of $q$ consists of points $\y \in\R^2$ for which $\Delta\le 0$.

 	Now assume $y_1=0$. From $q_{\mathrm{im}}=0$ we can observe that 
 	$\mathbf{0} \not\in \mathcal{I}(q)$. Thus, assume $y_2\neq 0$. Solving $q_{\mathrm{re}}=0$ for $x_2$ and substituting in $q_{\mathrm{im}}=0$ implies
 	$x_1^4-y_2^2(4y_2^2 - 1)=0.$	
 	This equation has a real solution if and only if $-y_2^2(4y_2^2 - 1)\le0$. Substituting $y_1=0$ in $\Delta$ allows to write $\Delta$ in terms of $y_2$, which gives 
 	$\Delta_{y_2} = -y_2^2(4y_2^2 - 1)^3.$
 	Therefore, the imaginary projection on the $y_2$-axis is $\{(0,y_2)\in\R^2 :\Delta_{y_2}\le 0\}\setminus\{(0,0)\}$. Thus,
 	\[
 	\mathcal{I}(q)=\{\y \in\R^2:-64y_1^8 - 128y_1^4y_2^4 - 64y_2^8 - 80y_1^4y_2^2 + 48y_2^6 + y_1^4 - 12y_2^4 + y_2^2\le 0\}\setminus\{\mathbf0\}.
 	\]
 	
 	\medskip
 	
 The irreducibility of the polynomial above over $\C$ can
 be verified for example using {\sc Maple}. For the original polynomial $p$, this gives the inequality description for 
 	$\mathcal{I}(p)$ stated in~\eqref{eq:example1} in the Introduction. 
 \end{example} 	 
	
	Even in the case of real polynomials, extending the case of real conics by letting the degree or the number of variables be greater than two  dramatically increases the difficulty of characterizing the imaginary projection. Let us see one such example of a cubic plane curve, i.e., where we have two unknowns and the total degree is three.

	\begin{example}
		Let $p\in\R[\z] = \R[z_1,z_2]$ be of the form $p = z_1^3 + z_2^3 - 1$. The similar attempt as before to calculate the imaginary projection $\I(p)$ is to separate the real and the imaginary parts of $p$ according to \eqref{PolySystem},
		\[
		p_{\rm re} = x_1^3 - 3x_1y_1^2 + x_2^3 - 3x_2y_2^2 - 1=0 \;
		\text{ and } \;
		p_{\rm im} = 3x_1^2y_1 + 3x_2^2y_2 - y_1^3 - y_2^3    =0.
		\]

		Despite the simplicity of the polynomial $p$, one cannot use the previous techniques to find the values of $\y \in\R^2$ such that the above system has real solutions for $\x$. The reason is that both $x_1$ and $x_2$ appear in higher degree than one in both equations. The resultant with respect to one of $x_1$ or $x_2$ is a univariate polynomial of degree six in the other, where we lack the exact tools to specify the reality of the roots.
	\end{example}	

		In the following theorem, we show that the imaginary projection of a generic complex plane curve of odd degree is
	 the whole plane. 
		
		\begin{thm}\label{th:EvenDeg}
			Let $p\in\C[z_1,z_2]$ be a complex bivariate polynomial of total degree $d$ such that its initial form has no real roots in $\P^1$. If $d$ is odd then the imaginary projection $\mathcal{I}(p)$ is $\R^2$.
As a consequence, the imaginary projection of a generic
complex bivariate polynomial of odd total degree is $\R^2$.  
		\end{thm}
		\begin{proof}
			
			Since the initial form has no real roots, it can be written in the form
			\[
			\init(p) = \prod_{j=1}^{d}(z_1-\alpha_j z_2),
			\]
			where $\alpha_j\notin\R$ for $1\le j \le d$. 		
			Substitute $z_j=x_j+ {\rm i}y_j$ for $j=1,2$ in $p$ and form the polynomial system $p_{\rm re} = p_{\rm im} = 0$  as introduced in (\ref{PolySystem}). For any fixed $\y \in\R^2$, both equations are of total degree $d$ in $x_1$ and $x_2$. Denote by $p_{\mathrm{re}}^h$ and $p_{\mathrm{im}}^h$, the homogenization of these two polynomials by a new variable $x_3$.
			Since both, $p_{\mathrm{re}}^h$ and $p_{\mathrm{im}}^h$, have odd degree, the number of complex intersection points (counted with multiplicities) is odd while the number of non-real intersection points (counted with multiplicities) is even. Thus, there is a real intersection point 
in $\P^2_{\R}$.
			We claim that this intersection point lies in the affine piece where $x_3 = 1$.  This implies that for any given $\y \in\R^2$, there exist $x_1,x_2\in\R$ for which $p_{\mathrm{re}}=p_{\mathrm{im}}=0$ and therefore completes the proof.
			
			To prove our claim, we show that the two curves defined by $p_{\mathrm{re}}^h=0$ and $p_{\mathrm{im}}^h=0$ do not intersect at infinity, i.e., their intersection point has $x_3\neq 0$. Let us assume that they intersect at infinity and set $x_3 = 0$ in $p_{\mathrm{re}}^h$ and $p_{\mathrm{im}}^h$. This substitution turns the complex polynomial $p_{\mathrm{re}}^h+ {\rm i}p_{\mathrm{im}}^h$ into 
			\[
			q:=\prod_{j=1}^{d}(x_1-\alpha_j x_2).
			\]	
			
			Thus, for the two projective curves to intersect at infinity we need to have $q=0$. Since $\alpha_j\notin\R$ for $1\le j \le d$, the only real solution for $x_1$ and $x_2$ is zero. This is a contradiction.
		\end{proof}
	
	\begin{cor}
	  \label{co:wholeplane2}
		Let $p\in\C[z_1,z_2]$ be a complex bivariate polynomial. The imaginary projection $\I(p)$ is $\R^2$ if $p$ has a factor $q$ such that the total degree of $q$ is odd and its initial form has no real roots in $\P^1$.
	\end{cor}
		\begin{proof}
			Since for $p_1,p_2 \in \C[\z]$, we have $\I(p_1 \cdot p_2)
			= \I(p_1) \cup \I(p_2)$, we claim that if there is a factor $q$ in $p$ whose imaginary projection is $\R^2$, then $\I(p)=\R^2$. The result now follows from the previous theorem.
		\end{proof}
	
		In the following section, instead of the dimension we set the degree to be two and characterize the imaginary projection for a 
certain family of quadratic hypersurfaces.

		%%%%%%%%%%%%%%%%%%%%%%%%%%%%%%%%%%
		%%%%%%%   Quadratics with hyperbolic initial form  %%%%%%%
		%%%%%%%%%%%%%%%%%%%%%%%%%%%%%%%%%%
		%%%%%%%%%%%%%%%%%%%%%%%%%%%%%%%%%%

		\section{Complex quadratics with hyperbolic initial form\label{se:QuadraticsWithHyperbolicInit}}

As we have seen in Example \ref{ex:caseB}, the methods used to compute the imaginary projection of real quadratics is not always useful for complex ones. However, for a certain family, namely the quadratics with hyperbolic initial form, one can build up on the methods for the real case. To classify the imaginary projections of any family of polynomials, Lemma \ref{le:group-actions-improj} suggests bringing them to their proper normal forms. 

\begin{lemma}\label{le:hyp-Init-quadratic-forms}
Under the action of $G_n$, any quadratic polynomial $p\in\C[z_1, \dots, z_n]$ with hyperbolic initial form can be transformed to one of the following normal forms:\vspace{2mm}

\begin{enumerate}
	\item $z_1^2+\alpha z_2+ r z_3+\gamma$,\\
	\item $\sum_{i=1}^{j}z_i^2 - z_{j+1}^2+\alpha z_{j+2}+r z_{j+3}+\gamma \qquad\text{for some } j=1,\dots,n-1$,\vspace{3mm}
\end{enumerate} 

\noindent such that terms containing $z_k$  do not appear for $k>n$, and $\alpha,r,\gamma\in\C$.
\end{lemma}
\begin{proof}
	The initial form $\init(p)$ is a hyperbolic polynomial of degree two. That is, after a real linear transformation it can be either $z_1^2$ or of the form $\z'^TM\z'$ such that $\z' = (z_1,\dots, z_{j+1})$ for some $1\le j\le n-1$ and $M$ is a square matrix of size $j+1$ with signature $(j,1)$. See \cite{garding-59}. This explains the initial forms in (1) and (2). 
	
Any term $\lambda z_j$ for some $1\le j \le n$, such that $z_j$ appears in our transformed initial forms, cancels out by one of the translations $z_j\mapsto z_j\pm \frac{\lambda}{2}$ without changing the initial form. Finally, we show that the number of linear terms in the rest of the variables is at most two. Consider the complex linear form $\sum_{j=1}^{n}\lambda_j z_j$. 
For $1\le j \le n$, let $\lambda_j= r_j + {\rm i} s_j$ such that $r_j,s_j\in\R$. We can now write the sum as 
$
(\sum_{j=1}^{n}r_j z_j)+{\rm i}(\sum_{j=1}^{n}s_j z_j)
$. If in the real part at least one of the $r_j$, say, $r_1$, is non-zero, then a
sequence of linear transformations $z_1\mapsto z_1-\frac{r_j}{r_1}z_j$ for $j=2,\dots,n$, cancels out $\sum_{j=2}^{n}r_j z_j$. Similarly, the complex part reduces to only one term. 
\end{proof}

We first focus on the case where $n=2$. In this case, we explicitly express the unbounded spectrahedral components forming $\I(p)^\compl$. The following subsection covers part of the proof of Theorem \ref{th:complex-classification1}.

%%%%%%%%%%%%%%%%%%%%%%%%%%%%%%%%%%
%%%%%%   Complex conics with hyperbolic initial form  %%%%%%
%%%%%%%%%%%%%%%%%%%%%%%%%%%%%%%%%%
%%%%%%%%%%%%%%%%%%%%%%%%%%%%%%%%%%

\subsection{Complex conics with hyperbolic initial form}\label{subs:Conic-Hyp-Init}
To match them with our classification of conics in Theorem~\ref{th:conic-classification1}, we do a real linear transformation in the case (2) and write them as
\[
\begin{matrix}
\text{(1a.1)} \,p=z_1^2+\gamma,&&&&\text{(1a.2)}\,p=z_1^2+\gamma z_2\,\,\,\,\,\,\gamma\neq 0,&&&&\text{(1b)} \,p=z_1z_2+\gamma,
\end{matrix}
\]
\noindent for some $\gamma\in\C$. 
To find $\I(p)$ for each normal form, we compute the resultant of the two real polynomials, as introduced in (\ref{PolySystem}), with respect to $x_i$ to have a univariate polynomial in $x_j$, where $i,j\in\{1,2\}$, and $i\neq j$. Then we use the discriminantal conditions on the univariate polynomials to argue about the real roots.

First consider the normal form (1a.1). If $\gamma_{\mathrm{im}}=0$, then we have the real conics of the cases $(vi)$ and $(viii)$ in Theorem \ref{th:RealConicChar}.
The two real polynomials  $
p_{\mathrm{re}} = x_1^2-y_1^2+\gamma_{\mathrm{re}} =0
\; \text{ and } \;
p_{\mathrm{im}} = 2 x_1 y_1+\gamma_{\mathrm{im}} =0
$ form the system (\ref{PolySystem}) here.
From $\gamma_{\mathrm{im}}\neq 0$, we need to have $y_1\neq 0$. Now substituting $x_1 = \frac{-\gamma_{\mathrm{im}}}{2 y_1}$ from $p_{\mathrm{im}}=0$ into $p_{\mathrm{re}}=0$ and solving for $y_1^2$ implies
$
y_1^2 = \frac{1}{2}\left(\gamma_{\mathrm{re}}+\sqrt{\gamma_{\mathrm{re}}^2+\gamma_{\mathrm{im}}^2}\right).
$
Therefore,\vspace{-2mm}
\begin{equation}
\tag{\text{1a.1}}
\mathcal{I}(p) = 
\begin{cases}
\text{A unique line} &\text{if } \gamma\in\R_{\le 0},\\
\text{Two parallel lines} &\text{otherwise}.
\end{cases}
\end{equation}

Clearly, the closures of the components in the complement are spectrahedra.
\medskip

Now consider (1a.2) which is a generalization of the parabola case $(iii)$ in Theorem~\ref{th:RealConicChar}, where $\gamma=1$. 
 Similarly to the previous case, we build the corresponding polynomial system as (\ref{PolySystem}). The discriminantal condition after substituting $x_2$ from $p_{\mathrm{im}} = 0$ into $p_{\mathrm{re}} = 0$ implies that there exists a real $x_1$
if and only if ${4|\gamma|^2(y_1^2+\gamma_{\mathrm{im}}y_2)\ge 0}$.
Hence, $\I(p)^\compl$ consists of $\y\in\R^2$ such that
$y_1^2+\gamma_{\mathrm{im}}y_2< 0$.
This inequality specifies the open subset of $\R^2$ bounded by the parabola $y_1^2+\gamma_{\mathrm{im}}y_2= 0$ and containing its focus.
Therefore,
\begin{equation}
\tag{\text{1a.2}}
\mathcal{I}(p) = 
\begin{cases}
\R^2\setminus \{(0,y_2):y_2\neq 0\} &\text{if } \gamma\in\R, \\
\{\y \in\R^2 : y_1^2+\gamma_{\mathrm{im}}y_2\ge 0\}&\text{otherwise.}
\end{cases}
\end{equation}

Notice that this incidence of $\I(p)^\compl$ consisting of one unbounded component does not occur for real conics. See Corollary \ref{cor:oneUnbdd}. Further, $\I(p)^\compl$ for $\gamma\notin\R$ is given by the unbounded spectrahedral set defined by 
\begin{equation*}
\left(\begin{matrix}
1 & y_1 \\
y_1 & -\gamma_{\mathrm{im}}y_2
\end{matrix}\right)\succ 0.
\end{equation*}

For the last case (1b) from the corresponding real polynomial system $p_{\mathrm{re}} =p_{\mathrm{im}}=0$, one can simply check that $\gamma=0$ implies $\mathcal{I}(p)=\{\y\in\R^2:y_1y_2=0\}$.
Now let $\gamma\neq 0$ and first assume $y_1y_2\neq 0$.
After the substitution of $x_2$ from $p_{\mathrm{im}}=0$ to $p_{\mathrm{re}}=0$, the discriminantal condition on the quadratic univariate polynomial to have a real $x_1$ implies
\[
\gamma_{\mathrm{re}} - |\gamma| \le 2 y_1 y_2 \le  \gamma_{\mathrm{re}}+ |\gamma|.
\]

If $\gamma\in\R\setminus\{0\}$, then $\mathbf{0}$ is the only point with $y_1y_2=0$ that is included in $\I(p)$. If $\gamma\notin\R$, then the union of the two axes except the origin is included in $\I(p)$. Thus,
\begin{equation}
\tag{\text{1b}}
\I(p)=
\begin{cases}
\text{The union of the two axes $y_1$ and $y_2$} & \text{if } \gamma = 0, \vspace{2mm}\\\vspace{2mm}
\left\{\y \in \R^2 \ : \ 0 < y_1 y_2 \le \gamma \right\}
\cup \{\mathbf{0}\}  & \text{if } \gamma \in \R_{>0}, \\
\vspace{2mm}
\left\{\y \in \R^2 \ : \ \gamma \le y_1 y_2 < 0 \right\}
\cup \{\mathbf{0}\} & \text{if } \gamma \in \R\setminus\R_{\ge0}, \\
\left\{\y \in \R^2 \ : \
\frac{1}{2}(\gamma_{\mathrm{re}} - |\gamma|) \le y_1 y_2 \le \frac{1}{2}( \gamma_{\mathrm{re}} + |\gamma| ) \right\}
\setminus \{\mathbf{0}\}  & \text{if } \gamma \not\in \R.   
\end{cases}
\end{equation}

\begin{corollary}\label{co:spectra-double-real-root}
	Let $p\in\C[z_1,z_2]$ be a complex conic with hyperbolic initial form.  The complement $\I(p)^\compl$
	of the imaginary projection consists of only unbounded spectrahedral components. 
\end{corollary}

\begin{proof}
	We saw this already for the cases (1a.1) and (1a.2). Therefore, we only prove the statement for (1b). There are four
	unbounded components, namely in each quadrant one, and
	no bounded component in $\mathcal{I}(p)^\compl$.
	The closures of the four unbounded components after setting 
	\[w = \sqrt{\frac{1}{2}( |\gamma|+\gamma_{\mathrm{re}})}\,\quad
	\text{and}\quad u = \sqrt{\frac{1}{2}( |\gamma| - \gamma_{\mathrm{re}})}\,\] 
	have
	the following representations as spectrahedra.
	In the quadrants $y_1y_2\geq 0$, they are expressed by ${y_1 y_2 - \frac{1}{2}(\gamma_{\mathrm{re}} + |\gamma|) \ge 0}$, or equivalently, $S_1(y_1,y_2)\succeq 0$ and ${ S_2(y_1,y_2)\succeq 0}$, where 
	\[
	S_1(y_1,y_2)=\left( \begin{array}{ccc}
	y_1 && w \\
	w & &y_2\\
	\end{array} \right),
	\quad
	S_2(y_1,y_2)=\left( \begin{array}{cc}
	-y_1 & w \\
	w & -y_2 \\
	
	\end{array} \right).
	\]
	
	In the quadrants with $y_1y_2\leq 0$, they are expressed by ${y_1 y_2 - \frac{1}{2} (\gamma_{\mathrm{re}} - |\gamma|) \le 0}$, or equivalently, $S_3(y_1,y_2)\succeq 0$ and $S_4(y_1,y_2)\succeq 0$, where
	
	\[
	S_3(y_1,y_2)=\left( \begin{array}{cc}
	y_1 & u \\
	u & - y_2 \\
	\end{array} \right),
	\quad 
	S_4(y_1,y_2)=\left( \begin{array}{ccc}
	- y_1 && u \\
	u & &y_2  
	\end{array} \right).\vspace{-6mm}	\]  \end{proof}

Given a conic $q$, an explicit description of the components of $\I(q)^\compl$ can be derived by using those of its normal form $p$ and applying on $\y$ the inverse operations turning $q$ to $p$.
We close this subsection by providing two examples for the cases (1a.2) and (1b) and their corresponding spectrahedral components.

\begin{example}\label{ex:(1a.2)}
	Let $q(z_1,z_2)=z_1^2+2z_1z_2+z_2^2+2{\rm i}z_2+1$. By applying the transformation $A$ and the translation $\w$ given by
	\[
	A:=\begin{pmatrix}
	1 & -1 \\
	0 & 1
	\end{pmatrix} \quad \text{and} \quad \w:=\begin{pmatrix}
	0 \\ \rm{i/2}
	\end{pmatrix},
	\]
	the conic $q$ is transformed to its normal form $p=z_1^2+2{\rm i}z_2$. Thus, we have
	{\small
		\[\I(p)^\compl=\left\{ y \in\R^2:\begin{pmatrix}
		1 & y_1 \\
		y_1 & -2y_2
		\end{pmatrix}\succ 0\right\}
		\ \text{and} \ \
		\mathcal{I}(q)^\compl=\left\{ y \in\R^2:\begin{pmatrix}
		1 & y_1+y_2 \\
		y_1+y_2 & -2y_2+1
		\end{pmatrix}\succ 0\right\},
		\] 	} \\
	such that $\I(q)^\compl$ is obtained by the inverse transformations for $\y$ in $\I(p)^\compl$.
	Figure \ref{fig:improjComplex} (1a) illustrates $\I(q)^\compl$.\vspace{-1mm}
\end{example}
\begin{example}\label{ex:1b}
	Let $q(z_1,z_2)=z_1^2-z_2^2+2{\rm i}$.  Applying $A=\frac{1}{2}\left(\begin{matrix}
	-1 & -1 \\
	-1 & 1
	\end{matrix}\right)$ transfers the conic $q$ into $p = z_1z_2+2{\rm i}=0$. The value of both $u$ and $w$ introduced in the proof of Corollary \ref{co:spectra-double-real-root} is 1. By applying $A^{-1}$ to $\y$, the matrices $S_1,\ldots,S_4$ transform to 
	{\small
		\begin{align*}
		T_1(y_1,y_2)=& \begin{pmatrix}
		-y_1-y_2 & 1 \\
		1 & -y_1+y_2 
		\end{pmatrix}, \ \qquad T_2(y_1,y_2)= \ \left(\begin{matrix}
		y_1+y_2 && 1 \\
		1 && y_1-y_2 
		\end{matrix}\right), \\ T_3(y_1,y_2)=& \ \begin{pmatrix}
		-y_1-y_2 &1 &\!\!\\
		1& y_1-y_2&\!\!
		\end{pmatrix}, \ \qquad T_4(y_1,y_2)= \ \begin{pmatrix}
		y_1+y_2 &1\, \\
		1& -y_1+y_2\,
		\end{pmatrix}.
		\end{align*} }	
	Thus, the complement of the imaginary projection as shown in Figure~\ref{fig:dist-real} is given by \[
	\overline{\I(q)^\compl}=\bigcup_{j=1}^4\left\{ \y \in\R^2:T_j(y_1,y_2)\succeq 0\right\}.
	\]\vspace{-4mm}
	\begin{figure}[H]
		\begin{subfigure}{.18\textwidth}
			\centering
			\includegraphics[width=1\linewidth]{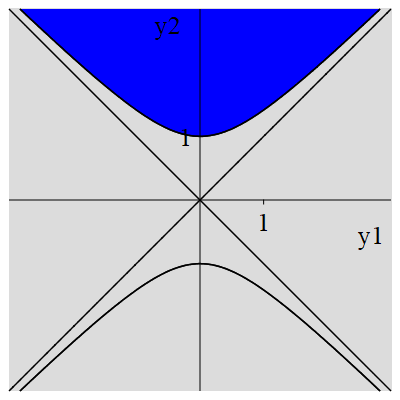}
		\end{subfigure}%
		\begin{subfigure}{.18\textwidth}
			\centering
			\includegraphics[width=1\linewidth]{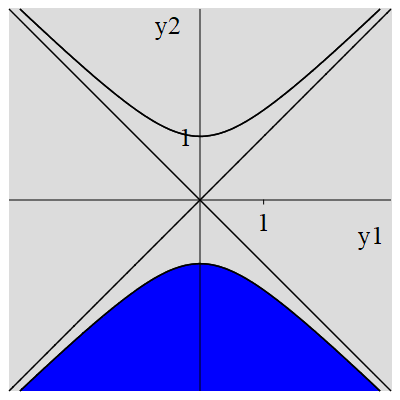}
		\end{subfigure} 
		\begin{subfigure}{.18\textwidth}
			\centering
			\includegraphics[width=1\linewidth]{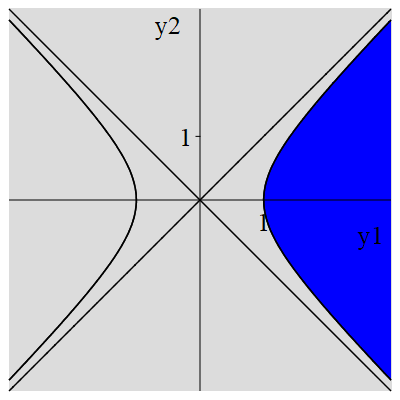}
		\end{subfigure}%
		\begin{subfigure}{.18\textwidth}
			\centering
			\includegraphics[width=1\linewidth]{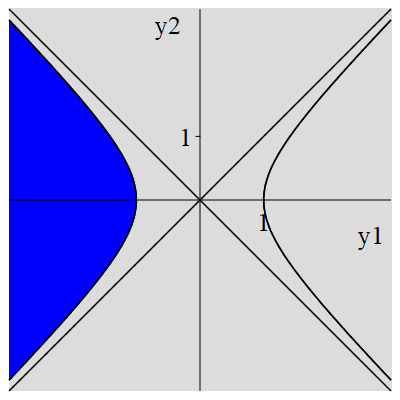}
		\end{subfigure} 
		\begin{subfigure}{.18\textwidth}
			\centering
			\includegraphics[width=1\linewidth]{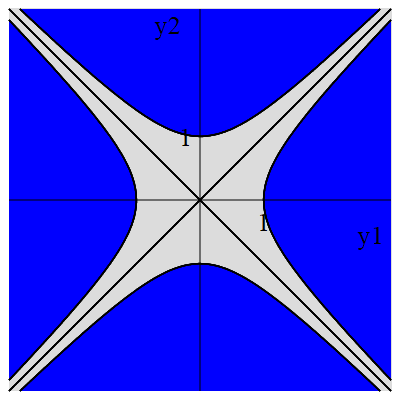}
		\end{subfigure}%
		\caption{The first four pictures represent $T_j(y_1,y_2)\succeq 0$ for $1\le j\le4$,  and the last one shows their union, which gives $\I(q)^\compl$ for $q =z_1^2-z_2^2+2{\rm i}$.}
		\label{fig:dist-real}
	\end{figure}
\end{example}

In the example above all four components are strictly convex, which can not occur in the case of real conics. This provides a key ingredient in the proof of Theorem \ref{th:StrictlyConvexComplex}.
%%%%%%%%%%%%%%%%%%%%%%%%%%%%%%%%%%%%%%%%%%%
%%%%%%%%%%%%%%%%%%%%%%%%%%%%%%%%%%%%%%%%%%%
%%%%%%%%%%%%Higher dimension complex quadratics%%%%%%%%%%%%%
%%%%%%%%%%%%%%%%%%%%%%%%%%%%%%%%%%%%%%%%%%%
%%%%%%%%%%%%%%%%%%%%%%%%%%%%%%%%%%%%%%%%%%%
\subsection{Higher dimensional complex quadratics}
We now let the dimension to be at least three and we use the normal forms provided in Lemma \ref{le:hyp-Init-quadratic-forms} to show the following classification of the imaginary projection. 
To avoid redundancy, for each quadratic polynomial we set $n$ to be the largest index of $z$ appearing in its normal form. Since we have already covered the case of conics, we need to consider $n\ge 3$.

\begin{theorem}
\label{th:ndimhyperbol1}
	Let $n \ge 3$ and $p\in\C[z_1,\dots,z_n]$ be a quadratic polynomial with hyperbolic initial form. Up to the action of $G_n$, the imaginary projection $\I(p)$ is either $\R^n$,  $\R^n\setminus\{(0,\dots,0,y_n)\in\R^n : y_n\neq 0\}$, or otherwise we can write $p$ as $p= \sum_{i=1}^{n-1}z_i^2 - z_{n}^2+\gamma$ for some $\gamma\in\C$ such that $|\gamma|=1$ and we get	
	\begin{equation*}
	\mathcal{I}(p)=
	\begin{cases}
	\left\{\y \in \R^n \ : \ y_n^2<\sum_{i=1}^{n-1} y_i^2 \right\}
	\cup \{\mathbf{0}\}  & \text{if } \gamma =1,\vspace{1mm} \\
	\vspace{1mm}
	\left\{\y \in \R^n \ : \ y_{n}^{2}-\sum_{i=1}^{n-1} y_i^2 \le 1 \right\}
	& \text{if } \gamma =-1, \\
	\left\{\y \in \R^n \ : \
	y_{n}^{2}-\sum_{i=1}^{n-1} y_i^2 \le\frac{1}{2}(1-\gamma_{\rm re}) \right\}
	\setminus \{\mathbf{0}\}  & \text{if } \gamma \not\in \R.  
	\end{cases}
	\end{equation*}
\end{theorem}

\begin{proof}
By real scaling and complex translations, any of the forms in Lemma \ref{le:hyp-Init-quadratic-forms} drops into one of the following cases:
\[
\begin{matrix}
\text{(a)}\, \alpha=r=\gamma=0,&&&
\text{(b)}\, \alpha=1,\,\, \text{and}\,\,  r=\gamma=0,&&&
\text{(c)}\,\alpha\notin\R,\,\, \text{and} \,\,r,\gamma=0,
\end{matrix}
\]\vspace{-6mm}\[
\begin{matrix}
\text{(d)}\, \alpha\notin\R,\, r=1,\,\, \text{and} \,\,\gamma=0,&&&
\text{(e)}\,\alpha=r=0, \,\,\text{and} \,\,\gamma\neq 0.	
\end{matrix}
\]

For the normal form (1) all cases but (d) drop into the conic sections discussed previously. Case (d) is similar for both normal forms (1) and (2). Thus we focus on (2). 

The imaginary projection for the cases (a) and (b) are known from the real classification and they are $\R^n$ and $ \R^n\setminus\{(0,\dots,0,y_n)\in\R^n : y_n\neq 0\}$, respectively. See \cite[Theorem 5.4]{jtw-2019}.

In case (c) after building the system (\ref{PolySystem}) and considering two cases, based on whether the real part of $\alpha$ is zero or not, one can then check that $\mathcal{I}(p) = \R^n$ as follows. We have
\[
\begin{array}{rcl}
p_{\rm re} &=&\sum_{i=1}^{n-2}x_i^2-x_{n-1}^{2}-\sum_{i=1}^{n-2}y_i^2+y_{n-1}^{2}+\alpha_{\rm re}x_n-\alpha_{\rm im}y_n,\\
p_{\rm im} &=&2\sum_{i=1}^{n-2}x_iy_i-2 x_{n-1} y_{n-1}+\alpha_{\rm im} x_{n}+\alpha_{\rm re} y_{n}.
\end{array}
\]

First assume $\alpha_{\rm re} = 0$. For any $\y\in\R^n$, the equation $p_{\rm re} = 0$ has solutions $(x_1,\dots,x_{n-1})\in\R^{n-1}$. By substituting any of those solutions in $p_{\rm im} = 0$ we can solve it for $x_n$ and get a real solution. Now let $\alpha_{\rm re}\neq 0$. In this case, we substitute $x_n$ from the second equation into the first. For any $\y\in\R^n$, we get $\sum_{i=1}^{n-2}(x_i-r_i)^2-(x_{n-1}-r_{n-1})^{2} = r_n$ for some $r_1,\dots,r_{n}\in\R$ and therefore, there always exists a real solution $(x_1,\dots,x_{n-1})\in\R^{n-1}$.

Similarly, in the case (d), for any $\y\in\R^n$, there exists a real solution $(x_1,\dots,x_{n-1})\in\R^{n-1}$ for $p_{\rm im} = 0$ and for any $\y\in\R^n$ and any $(x_1,\dots,x_{n-1})\in\R^{n-1}$, there exists a real $x_n$ for $p_{\rm re}=0$. Thus $\I(p) = \R^n$ in this case, too.

Now we focus on case (e).
Let $p= \sum_{i=1}^{n-1}z_i^2 - z_{n}^2+\gamma$ for some $\gamma\in\C \setminus \{0\}$. Building the real system (\ref{PolySystem}) for $p$ yields
\[
\begin{matrix}
p_{\rm re} &=&\sum_{i=1}^{n-1}x_i^2-x_{n}^{2}-\sum_{i=1}^{n-1}y_i^2+y_{n}^{2}+\gamma_{\rm re},&&
p_{\rm im} &=&2\sum_{i=1}^{n-1}x_iy_i-2 x_{n} y_{n}+\gamma_{\rm im}.
\end{matrix}
\]

We can  assume $|\gamma| = 1$. Note that $\{\mathbf{0}\}\in\mathcal{I}(p)$ if and only if $\gamma\in\R$. We can thus exclude the origin in the following calculations.  Moreover, Theorem \ref{th:real-Quad} shows the cases where 
$\gamma = \pm 1$. Thus, we need to consider the case $\gamma\notin\R$.
\medskip

Let $T$ be an orthogonal transformation on $\R^{n-1}$. Invariance of the polynomials $\sum_{j=1}^{n-1}{y}_j^2$ and $\sum_{j=1}^{n-1}x_jy_j$ under 
the mapping $(x,y) \mapsto (T(x),T(y))$ implies
\begin{center}
	$(y_1,y_2,\dots,y_{n}) \in \mathcal{I}(p)\qquad$ if and only if $\qquad (y'_1,\dots,y'_{n-1},y_n) \in \mathcal{I}(p)$,
\end{center}
where $(y'_1,\dots,y'_{n-1}) = T(y_1,\dots,y_{n-1})$.
For a given $\y\in\I(p)$, let $T$ be a transformation with the property
$T(y_1,\dots,y_{n-1})=(\sqrt{\sum_{i=1}^{n-1}y_i^2},0,\dots, 0)$ and set $(x'_1,\dots,x'_{n-1})=T(x_1,\dots,x_{n-1})$.
We can now rewrite the simplified polynomial system as
\[
\begin{matrix}
p_{\rm re} &=&\sum_{i=1}^{n-1}{x'_{i}}^{2}-x_{n}^{2}-{y'_1}^{2}+y_{n}^{2}+\gamma_{\rm re}, &&&
p_{\rm im} &=& 2 x'_{1} y'_{1}-2 x_{n} y_{n}+\gamma_{\rm im}.
\end{matrix}
\]

First consider $y'_1 = 0$. This implies $y_n\neq 0$. Solving $p_{\rm im} =0$ for $x_n$ and substituting in $p_{\rm re} = 0$ implies
\[
4y_{n}^2(\sum_{i=1}^{n-1}{x'_i}^{2})=\left(\gamma_{\rm re}^{2}+\gamma_{\rm im}^{2}\right)-\left(2y_{n}^{2}+\gamma_{\rm re}^{}\right)^{2} = 1-\left(2y_{n}^{2}+\gamma_{\rm re}\right)^{2}.
\]

This has a real solution for $(x'_1,\dots,x'_{n-1})$ if and only if $y_n^2\le\frac{1-\gamma_{\rm re}}{2}$. Now assume $y'_1 \neq 0$. Observe that if ${y'_1}^2 = y_n^2$ then we always get a real solution. Thus assume $\frac{y_{n}^{2}}{{y'_1}^2}-1\neq 0$.
Solving $p_{\rm im} =0$ for $x'_1$ and substituting in $p_{\rm re} = 0$ implies 
{\small\[
\left(\frac{y_{n}^{2}}{{y'_1}^2}-1\right) \left(x_{n}-\frac{\gamma_{\rm im} y_{n}}{2 {y'_1}^2 \left(\frac{y_{n}^{2}}{{y'_1}^2}-1\right)}\right)^{2}+\sum_{i=2}^{n-1}{x'_i}^{2}+ \frac{\left(y_{n}^{2}-{y'_1}^2\right)^{2}+\gamma_{\rm re}\left(y_{n}^{2}-{y'_1}^2\right)-\left(\frac{\gamma_{\rm im}^{}}{2}\right)^{2}}{ y_{n}^{2}- {y'_1}^2}=0.
\]}

If  ${y'_1}^2 > y_n^2$, there always is a real solution and otherwise, it has a real solution if and only if $\left(y_{n}^{2}-{y'_1}^{2}\right)^{2}+\gamma_{\rm re}\left(y_{n}^{2}-{y'_1}^{2}\right)-\left(\frac{\gamma_{\rm im}^{}}{2}\right)^{2} \le 0$. That is, $y_{n}^{2}-{y'_1}^{2}\le\frac{1-\gamma_{\rm re}^{}}{2}$.
To get the imaginary projection of the original system, it is enough to do the inverse transformation $T^{-1}$. This completes the proof.
\end{proof}

\begin{cor}
  \label{co:ndimhyperbol2} 
	Let $p\in\C[z_1,\dots,z_n]$ be a quadratic polynomial with hyperbolic initial form. Then
	\begin{itemize}
		\item[(1)] the complement $\mathcal{I}(p)^\compl$ is either empty or it consists of
		\subitem- one, two, three, or four unbounded components; or
		\subitem- two unbounded components and a single point.
		\item[(2)] the complement of the closure $\overline{\mathcal{I}(p)}^\compl$ is either empty or unbounded.
		\item[(3)] the algebraic degrees of the irreducible components in $\partial\I(p)$ are at most two. 
	\end{itemize}
\end{cor}

		%%%%%%%%%%%%%%%%%%%%%%%%%%%%%%%%%%
		%%%%%%%%%%%%%   Main Classification  %%%%%%%%%%
		%%%%%%%%%%%%%%%%%%%%%%%%%%%%%%%%%%
		%%%%%%%%%%%%%%%%%%%%%%%%%%%%%%%%%%

		\section{The main classification of complex conics\label{se:mainclassification}}

		In this section, we give a classification of the imaginary projection $\mathcal{I}(p)$ where $p\in\C[\z] = \C[z_1,z_2]$ is a complex conic as in Definition \ref{def:complexConic}. We state our topological classification in terms of the number and boundedness of the components in  $\mathcal{I}(p)^\compl$.
		In particular, this implies that the number of bounded and unbounded components do not exceed one and four, respectively. Furthermore, $\mathcal{I}(p)^\compl$ cannot contain both bounded and unbounded components for some complex conic $p$. 
		
A main achievement of this section is to establish a suitable classification and normal forms of complex conics under the action of the group $G_2$.
There are infinitely many orbits on the set of complex conics under this action, since
the real dimension of $G_2$ is $8$ and the set of complex conics has real dimension $10$. Each of our normal forms corresponds to infinitely many orbits that share their topology of imaginary projection by Lemma \ref{le:group-actions-improj}.

As a consequence of the obstructions in the existing classifications of conics that we discussed in the Introduction, we developed our own classification of conic sections. It is based on the five distinct arrangement possibilities for the roots of the initial form in $\P^1$ that are grouped in two main cases, depending on whether the initial form of the complex conic is hyperbolic or not: 
	
	\vspace{-2mm}
	\setlength{\columnsep}{-10pt}
	\begin{multicols}{2}
		\begin{itemize}
			\item[] \hspace{-0.82cm}\underline{Hyperbolic initial form}\vspace{3mm}
		\item[] \vspace{-1mm}
			\item[(1a)]  A double real root\vspace{2mm}
			\item [(1b)]Two distinct real roots\vspace{2mm}
			\item[] \hspace{-0.85cm}\underline{Non-hyperbolic initial form}\vspace{3mm}
			\hspace{-3mm}		\item [(2a)]A double non-real root\vspace{2mm}
			\item [(2b)]One real and one non-real root\vspace{2mm}
			\item [(2c)]Two distinct non-real roots\vspace{2mm}
		\end{itemize}
	\end{multicols}\vspace{-1mm}

	\begin{thm}[\textbf{Topological Classification}]
		\label{th:complex-classification1}
		Let $p\in\C[z_1,z_2]$ be a complex conic. For the above five
		 cases, the set $\I(p)^\compl$ is
			\vspace{-4mm}
		\setlength{\columnsep}{-10pt}
		\begin{multicols}{2}
			\begin{itemize}
				\item[] \vspace{2mm}
				\item[(1a)]  the union of one, two, or three  
							\item[]unbounded components.\vspace{2mm}
				\item [(1b)] the union of four
					\item[]unbounded components.\vspace{2mm}
				\item[] \vspace{2mm}
				\item [(2a)] empty. \vspace{2mm}
				\item[(2b)] empty, a single point, 
				\item[]or a line segment.\vspace{2mm}
				\item [(2c)]empty or one bounded component, 
				\item[]possibly open.\vspace{2mm}
			\end{itemize}
		\end{multicols}\vspace{-2mm}	
		In particular, the components of  $\mathcal{I}(p)^\compl$ are spectrahedral in all
		the first four classes. This is not true in general for the last class {\rm(2c)}. 
	\end{thm}

The following corollary relates the boundedness of the components in $\mathcal{I}(p)^\compl$ to the hyperbolicity of the initial form $\init(p)$.

	\begin{cor}\label{co:hyperbolicityBdd}
		Let $p\in\C[z_1,z_2]$ be a complex conic. Then $\mathcal{I}(p)^\compl$ consists of  unbounded components if and only if the initial form of $p$ is hyperbolic. Otherwise, $\mathcal{I}(p)^\compl$ is empty or consists of one bounded component. 
		Moreover, if there is a bounded component with 
		non-empty interior, then $\init(p)$ has two distinct non-real roots.
	\end{cor}
	Figure~\ref{fig:improjComplex} represents the types that do not appear for real coefficients. For instance, the middle picture, labeled as (2b), shows the case where  $\I(p)^\compl$ consists of a bounded component with empty interior. This can not occur if $p$ has only real coefficients.
The other two pictures are discussed in the next two corollaries. The following corollary compares the algebraic degrees of the irreducible components in the boundary $\partial\I(p)$. Its proof comes at the end of the next section.
	
	\begin{figure}
		\begin{subfigure}{.24\textwidth}
			\centering
			\includegraphics[width=.9\linewidth]{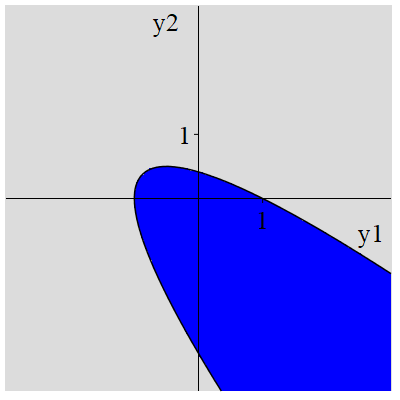}
			\caption*{(1a)}
		\end{subfigure}%
		\begin{subfigure}{.24\textwidth}
			\centering
			\includegraphics[width=.9\linewidth]{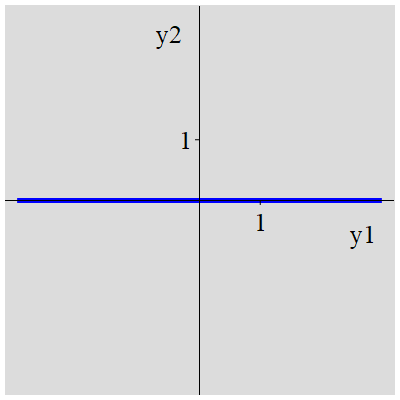}
			\caption*{(2b)}
		\end{subfigure} 
		\begin{subfigure}{.24\textwidth}
			\centering
			\includegraphics[width=.9\linewidth]{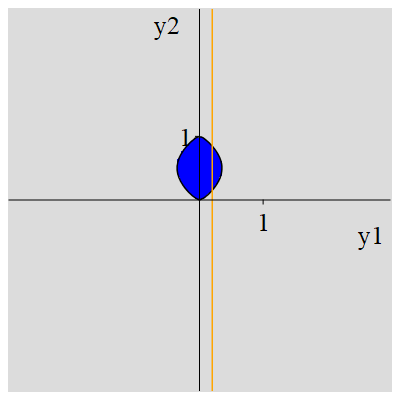}
			\caption*{(2c)}
		\end{subfigure}%
		\vspace{-2mm}	\caption{The complements of the imaginary projections are colored in blue. The
			pictures show cases in the classification of the imaginary
			projection for complex conics which do not appear for real conics.
The orange line in the right figure represents a generic line intersecting 
the boundary in two points, which is used to prove the
non-spectrahedrality of this example in Section~\ref{se:non-hyperbolic}.
			}\label{fig:improjComplex}
	\end{figure}
		\begin{cor}
			\label{co:alg-degrees}
			Let $p\in\C[z_1,z_2]$ be a complex conic.
			
			\begin{enumerate}
				\item The  boundary $\partial\mathcal{I}(p)$ may not be algebraic. The algebraic degree of any irreducible component in its Zariski closure is at most 8. The bound is tight. If $\mathcal{I}(p)^\compl$ has no bounded
				components, then $\partial\mathcal{I}(p)$ is algebraic and it consists of irreducible pieces of degree at most two.\vspace{2mm}
				\item If all coefficients are real, then $\partial\mathcal{I}(p)$ is algebraic and it consists of irreducible pieces of degree at most two.
			\end{enumerate}
		\end{cor}
		
Example \ref{ex:caseB}, that is shown in Figure \ref{fig:improjComplex} (2c), illustrates an instance where the above contrast appears. The next corollary compares the number and strict convexity of the unbounded components that occur in $\mathcal{I}(p)^\compl$ when $p$ is a complex or a real conic.
		
		\begin{cor}\label{cor:oneUnbdd}
			Let $p\in\C[z_1,z_2]$ be a complex conic.
			
			\begin{enumerate}
				\item The number of unbounded components in $\mathcal{I}(p)^\compl$ can be any integer $0\le k\le 4$ and up to 4 of them can be strictly convex.  \vspace{2mm}
				\item If all coefficients are real, the number of unbounded components in $\mathcal{I}(p)^\compl$ can be any integer $0\le k\le 4$ except for $k=1$ and up to 2 of them can be strictly convex. \vspace{2mm}
			\end{enumerate}
		\end{cor}
The proof follows from Theorems \ref{th:RealConicChar} and \ref{th:complex-classification1}, together with Example \ref{ex:1b}. The highlighting difference in the previous corollary, i.e., when $\I(p)^\compl$ has one unbounded component, appears in the first class (1a) where the initial form has a double real root. Example \ref{ex:(1a.2)} provides such an instance and is shown in Figure \ref{fig:improjComplex} (1a).	

\medskip

Theorem~\ref{th:complex-classification1} is only proven by the end of Section \ref{se:non-hyperbolic}. In the previous section, we discussed the case where $p$ has hyperbolic initial form in details.  
It remains to consider the case where $\init(p)$ is not hyperbolic. As in Subsection \ref{subs:Conic-Hyp-Init}, we first need to compute proper normal forms and then by Lemma~\ref{le:group-actions-improj}, it suffices to compute the imaginary projections of those forms for each case.

	\begin{thm}[\textbf{Normal Form Classification}]
		\label{th:conic-classification1}
		With respect to the group $G_2$, there are infinitely many orbits for the complex conic sections with the following representatives.\vspace{-6mm}
		\setlength{\columnsep}{-20pt}
		\begin{multicols}{2}
			\begin{itemize}
				\item[]
				\item[]
				\item[(1a)]	$\begin{array}{l}
				{\rm(1a.1)}\,\,p=z_1^2+\gamma\vspace{2mm}\\
				{\rm(1a.2)}\,\	p=z_1^2+\gamma z_2
				\end{array}$
				\item[]
				\item[(1b)] $p=	z_1z_2+\gamma$ 
				\item[]
				\item[]
				\item[(2a)] $\begin{array}{l}
				{\rm(2a.1)}\,\,p=(z_1-{\rm i} z_2)^2 + \gamma\vspace{2mm}\\
				{\rm(2a.2)}\,\	p = (z_1 - {\rm i}z_2)^2 + \gamma z_2
				\end{array}$
				\item[]
				\item[(2b)] $p = z_2 (z_1 - \alpha z_2) + \gamma$\vspace{2mm}	
				\item[(2c)] $\begin{array}{l}
				{\rm(2c.1)}\,\,p = z_1^2+z_2^2+\gamma\vspace{2mm}\\
				{\rm(2c.2)}\,\	p=(z_1 - {\rm i} z_2)(z_1 - \alpha z_2)+\gamma
				\end{array}$				
			\end{itemize}
		\end{multicols}\vspace{-1mm}	
\noindent for some $\gamma,\alpha\in\C$ such that, to avoid overlapping, we assume $\gamma\neq 0$ in {\rm(1a.2)} and {\rm(2a.2)}, $\alpha\notin\R$ in {\rm(2b)} and {\rm(2c.2)}, and finally $\alpha\neq \pm {\rm i}$ in \rm{(2c.2)}.
	\end{thm}	
	\begin{proof}
				By applying a real linear transformation we first map the roots of $\init(p)$ to $(0:1)$ in (1a),
		 to $(1:0)$ and $(0:1)$ in (1b), to $({\rm i : 1})$ in (2a), to $(1:0)$ and $(\alpha,1)$ such that $\alpha\notin\R$ in (2b), to $(\pm{\rm i} : 1)$ in (2c.1), and to $({\rm i} : 1)$ and $(\alpha : 1)$ such that $\alpha\notin\R$ and $\alpha\neq\pm{\rm i}$ in (2c.2). Then, similar to the proof of Lemma \ref{le:hyp-Init-quadratic-forms}, by eliminating some linear terms or the constant by complex translations we arrive at the given normal forms for each case. 		 
			Since the arrangements of the two roots in $\P^1$ is invariant under the action of $G_2$, the given five cases lie in different orbits. Note that the orbits of the subcases in each case do not overlap. For the subcases of (1a), in (1a.2), $z_1$ and $z_2$ may be transformed to $az_1+bz_2+e$ and $cz_1+dz_2+f$ with $a,b,c,d\in\R$ and $e,f\in\C$. This leads to $(az_1+bz_2+e)^2+\gamma$. Since $z_2^2$ does not appear in the normal form of case (1a.2), we get $b=0$ and thus $z_2$ can not appear. Further $z_1^2+\gamma_1$ and $z_1^2+\gamma_2$ with $\gamma_1\neq \gamma_2$ belong to different orbits since the previous argument enforces $a=1,b=0,e=0$. The other cases are similar. Thus, for any of the eight normal forms, there are infinitely many orbits corresponding to each $\gamma\in\C$ (and $\alpha\in\C$ in some cases).
	\end{proof}
		
		%%%%%%%%%%%%%%%%%%%%%%%%%%%%%%%%%%%%
		%%%%%%       Complex conics with non-hyperbolic initial     %%%%%
		%%%%%%%%%%%%%%%%%%%%%%%%%%%%%%%%%%%%
		%%%%%%%%%%%%%%%%%%%%%%%%%%%%%%%%%%%%

		\section{Complex conics with non-hyperbolic initial form\label{se:non-hyperbolic}}
		
		We complete the proof of the Topological Classification 
		Theorem~\ref{th:complex-classification1} by treating the case where the complex
		conic $p \in \C[\z] = \C[z_1,z_2]$ does not have a hyperbolic initial form.
		In particular, we see that, as previously stated in Corollary \ref{co:hyperbolicityBdd}, if the initial form of $p$ is not hyperbolic, then $\mathcal{I}(p)^\compl$ is empty or consists of one bounded component whose interior is non-empty only if $\init(p)$ has two distinct non-real roots in $\P^1$. 
		
		The overall steps in computing the imaginary projection of the cases with non-hyperbolic initial form are as follows. After building up the real polynomial system for the classes (2b) and (2c.1) of
 Theorem~\ref{th:conic-classification1}
 as in~\eqref{PolySystem}, we use the same techniques as in Subsection \ref{subs:Conic-Hyp-Init}. However, in the case (2a), by the nature of the polynomial system, we directly argue that the imaginary projection is $\R^2$. In the last case (2c.2), we do not explicitly represent the components of $\I(p)^\compl$. Instead, in Theorem~\ref{OnebddComp}
		we prove that it does not contain any unbounded components and the number of bounded components does not exceed one.
		
		%%%%%%%%%%%%%%%%%%%%%%%%%%%%%%%%%%%%
		%%%%%%%%%%%  A double non-real   %%%%%%%%%%%%%%%
		%%%%%%%%%%%%%%%%%%%%%%%%%%%%%%%%%%%%
		%%%%%%%%%%%%%%%%%%%%%%%%%%%%%%%%%%%%
		
		\subsection{A double non-real root (2a)} We show that in this case we have a full space imaginary projection.
		First consider the normal form (2a.1). We have
		\[	\begin{matrix}
			p_{\mathrm{re}} & = &
			x_1 ^2- x_2^2+ 2  y_2 x_1+ 2 y_1 x_2  + 
			\gamma_{\mathrm{re}}x_2-y_1^2  + y_2^2 -\gamma_{\mathrm{im}}y_2&=& 0, \\
			p_{\mathrm{im}} & = &
			- 2 x_1 x_2 +2 y_1x_1  - 2y_2 x_2 + \gamma_{\mathrm{im}}x_2 + 2 y_1 y_2  + \gamma_{\mathrm{re}}y_2&= &0.
		\end{matrix}\]
		
		We prove $\mathcal{I}(p) = \R^2$ by showing that for every given $\y \in \R^2$, 
		these two real conics in $\x=(x_1,x_2)$ have a real intersection point. For any fixed $\y \in \R^2$, the bivariate polynomial 
		$p_\mathrm{re}$ in $\x$
		has the quadratic part $x_1^2-x_2^2$, and hence, the equation $p_{\mathrm{re}}=0$ defines
		a real hyperbola in $\x$ with 
		asymptotes
		$x_1=x_2+c_1$ and $x_1=-x_2+c_2$ for some constants $c_1,c_2 \in \R$; possibly the hyperbola
		degenerates to  
		a union of these two lines.
		The degree two part of the polynomial $p_{\mathrm{im}}$ is given by $-2x_1x_2$
		and hence, the equation $p_{\mathrm{im}}=0$
		defines a real hyperbola in $\x$ with asymptotes
		$x_1= d_1$ and $x_2=d_2$ for some constants $d_1, d_2 \in \R$;
		possibly the hyperbola may degenerate to a union of these two lines.
		Since the two hyperbolas have a real intersection point, the
		claim follows. The case (2a.2) is similar.

		%%%%%%%%%%%%%%%%%%%%%%%%%%%%%%
		%%%%%%%%     one real one non-real   %%%%%%%%%
		%%%%%%%%%%%%%%%%%%%%%%%%%%%%%%
		%%%%%%%%%%%%%%%%%%%%%%%%%%%%%%
		
		\subsection{One real and one non-real root (2b)}\label{subs:real-non-real}
		
		This case gives the system of 
		equations
		\[	\begin{matrix}
			p_{\mathrm{re}}  & = & - \alpha_{\mathrm{re}} x_2^2+ x_1 x_2 + 2 \alpha_{\mathrm{im}} y_2 x_2 + \alpha_{\mathrm{re}} y_2^2 - y_1 y_2  
			+ \gamma_{\mathrm{re}} & = & 0, \\
			p_{\mathrm{im}} & = &	- \alpha_{\mathrm{im}} x_2^2+ y_2 x_1+ y_1x_2   -2 \alpha_{\mathrm{re}}y_2 x_2 
			+ \alpha_{\mathrm{im}} y_2^2 + \gamma_{\mathrm{im}} & =& 0.
		\end{matrix}
		\]
		
		First assume $y_2\neq 0$. By solving the second equation for $x_1$, substituting the solution into the first equation and clearing the denominator, we get a univariate cubic polynomial in $x_2$ with non-zero leading coefficient. Since real cubic polynomials always have a real root,
		this shows that for $\y \in \R^2$ with $y_2 \neq 0$, there is a solution $\x \in \R^2$.
		
		It remains to consider $y_2 = 0$. 
		In this case, the second equation has a real solution in $x_2$
		whenever the corresponding discriminant $y_1^2 + 4 \alpha_{\mathrm{im}} \gamma_{\mathrm{im}}$ is non-negative, and if one of these solutions
		is non-zero, the first equation then gives a real solution for $x_1$.
		The special case that in the second equation both solutions for $x_2$ are 
		zero, can only occur for $y_1 = 0$ and $\gamma_{\mathrm{im}} = 0$. Then the first
		equation has a real solution for $x_1$ if and only if $\gamma_{\mathrm{re}} = 0$.
		Altogether, we obtain
		\begin{equation*}
		\tag{\text{2b}}
				\mathcal{I}(p) \ = \ \begin{cases}
		\R^2 & \text{ if } \gamma=0 \ \ \text{or} \ \ \alpha_{\mathrm{im}}\gamma_{\mathrm{im}} > 0, \\
		\R^2 \setminus \{\mathbf{0}\} & \text{ if } \gamma\in\R\setminus\{0\},\\
		\R^2 \setminus \{(y_1,0) \ : \ y_1^2 < -4 \alpha_{\mathrm{im}} \gamma_{\mathrm{im}}\} &
		\text{ if } \alpha_{\mathrm{im}} \gamma_{\mathrm{im}} < 0.
		\end{cases}
		\end{equation*}

		Note that when $\gamma\in\R\setminus\{0\}$ then $\I(p)$ is open but not $\R^2$. This answers Question \ref{que:open-close}.
		See Figure \ref{fig:improjComplex} (2b) for the imaginary projection of $p = z_2(z_1-{\rm i} z_2)-{\rm i}$ from this class.

		%%%%%%%%%%%%%%%%%%%%%%%%%%%%%%%
		%%%%%%%%       conjugate roots           %%%%%%%%%
		%%%%%%%%%%%%%%%%%%%%%%%%%%%%%%%
		%%%%%%%%%%%%%%%%%%%%%%%%%%%%%%%
		
		\subsection{Two distinct non-real roots (2c)}\label{subs:2complex}
		First we show that in (2c.1), i.e., where the roots of the initial form are complex conjugate, the imaginary projection is one open bounded component. After forming the polynomial system (\ref{PolySystem}), the same methods as those in Subsection \ref{subs:Conic-Hyp-Init}, i.e., taking the resultant of the two polynomials $p_{\rm re}$ and $p_{\rm im}$ with respect to $x_2$ and checking the discriminantal conditions to have a real $x_1$, lead to the imaginary projection
		\begin{equation*}\label{k-disc}
		\tag{\text{2c.1}}
\I(p) = \Big\{\y\in\R^2 : y_1^2+y_2^2\ge \frac{1}{2}(\gamma_{\mathrm{re}}+\sqrt{\gamma_{\mathrm{re}}^2+\gamma_{\mathrm{im}}^2})\Big\}.
\end{equation*}
In particular, we have  $\I(p) = \R^2$ if and only if $\gamma_{\mathrm{im}}=0$ and $\gamma_{\mathrm{re}}\le0$. Hence, in the case of two non-real conjugate roots,
$\mathcal{I}(p)^{\compl}$ consists of either one or zero bounded component
and it is a spectrahedral set. 

		%%%%%%%%%%%%%%%%%%%%%%%%%%%%%%%%%%%
		%%%%%%%%   2 distinct non-real non-conjugate    %%%%%%%%%
		%%%%%%%%%%%%%%%%%%%%%%%%%%%%%%%%%%%
		%%%%%%%%%%%%%%%%%%%%%%%%%%%%%%%%%%%

The subsequent lemma shows that for the case (2c) in general
 $\mathcal{I}(p)^\compl$  is either empty or consists of one 
bounded component. 

		\begin{lemma}\label{OnebddComp}
		Let $p =  (z_1 - \alpha z_2) (z_1 - \beta z_2) + d z_1 + e z_2 + f$
		with $\alpha, \beta \not \in \R$ and $d,e,f \in \C$.
		Then
		\begin{enumerate}
		\item $\mathcal{I}(p)^{\compl}$ has at most one bounded component.
		\item $\mathcal{I}(p)^{\compl}$ does not have unbounded components.
		\end{enumerate}
		\end{lemma}

		\begin{proof}
		(1) Assume that there are at least two bounded components in $\mathcal{I}(p)^{\compl}$. By Lemma~\ref{le:group-actions-improj},
		we can assume without loss of generality that the 
		$y_1$-axis intersects both components. 
			Solving $p=0$ for $z_1$ gives
			{\small
				\begin{equation}
					\label{eq:one-component-branch1}
					z_1 \ = \ \frac{\alpha + \beta}{2}z_2
					           - \frac{d}{2} 
					+ \sqrt[\C]{\left( \frac{\alpha-\beta}{2} \right)^2 z_2^2 - e z_2 -f } \, .
				\end{equation}
			}
			By letting $z_2\in\R$
			we obtain two continuous 
			branches $y_1^{(1)}(z_2)$ and $y_1^{(2)}(z_2)$
			satisfying~\eqref{eq:one-component-branch1}.
			Therefore, the set  
			$\mathcal{I}(p) \cap \{\y \in \R^2 \, : \, y_2 = 0\}$ has at most two connected components.
			This is a contradiction to our assumption that the $y_1$-axis intersects the two bounded components in $\mathcal{I}(p)^\compl$.

For (2), assume that there exists an unbounded component in the complement of $\mathcal{I}(p)$. The convexity implies that it must contain a ray. By Lemma~\ref{le:group-actions-improj},
we can assume without loss
of generality that the ray is the non-negative part of the $y_1$-axis. 
Similarly to the proof of (1), we set $y_2 = 0$ and check
the imaginary projection on $y_1$-axis, using the two complex solutions
in~\eqref{eq:one-component-branch1}. 
			Since $\alpha \neq \beta$, we have
			$D:= 
			\left( \frac{\alpha-\beta}{2} \right)^2
			\neq 0$, where	$D$ is the discriminant of $\init(p)$ with
			$z_2$ substituted to 1. We consider two cases:  $D \not\in \R_{>0}$ and  $D \in \R_{>0}$.
			In both cases we get into a contradiction to the assumption that the unbounded component contains the non-negative part of the $y_1$-axis. 
			
			First assume $D \not\in \R_{>0}$. 
			For $z_2 \to \pm \infty$,
			the imaginary part of the radicand is dominated by the imaginary part of
			the square root of $D$.  Since $D \not\in \R_{>0}$ at least one of the two expressions\vspace{-2mm}
			{\small
				\begin{equation*}
					\label{eq:dominating1}
					\left(\frac{\alpha+\beta}{2}\right)_{\mathrm{im}} \pm
					\sqrt{\frac{-D_{\mathrm{re}}+\sqrt{D_{\mathrm{re}}^2+D_{\mathrm{im}}^2}}{2}}\,\,
			\end{equation*}}
			is non-zero. Thus, letting $z_2\mapsto\pm\infty$, implies $y_1\mapsto+\infty$ in at least one of the branches. 
			
			Now assume $D \in \R_{>0}$. This implies $(\alpha - \beta)/2 \in \R$. Thus $(\alpha + \beta)/2 \notin \R$, since otherwise it contradicts with $\alpha,\beta\notin\R$. In this case, by letting $z_2$ grow to infinity, the dominating expression for $y_1$ is 
			$\frac{1}{2}(\alpha+\beta)_{\rm im}z_2.$
			Therefore, $y_1$ converges to $+\infty$ in one of the two branches.
				In both cases, for some $s>0$, the ray $\{(y_1,0)\in\R^2:y_1\ge s\}$ lies in the imaginary projection. This completes the proof.		
		\end{proof}

	Before, in Example \ref{ex:caseB} we have shown that the defining polynomial of the imaginary projection can be irreducible of degree 8. The previous lemma enables us to 
	show that $\mathcal{I}(q)^\compl$
	has exactly one bounded component. Note that $\mathbf{0}\in\mathcal{I}(q)^\compl$. Let $B_\epsilon$ be an open ball with center at the origin and radius $\epsilon$. By letting $y_1$ and $y_2$ converge to zero, the dominating part of $\Delta$ is $y_1^4+y_2^2$. Thus, for sufficiently small $\epsilon$, any non-zero point in $B_\epsilon$ has $\Delta>0$. Therefore,  $\mathcal{I}(q)^\compl$ contains an open ball around the origin.
	Now the claim follows from Theorems \ref{OnebddComp}.

			In this example, the imaginary projection is Euclidean closed, i.e., $\overline{\mathcal{I}(q)}=\mathcal{I}(q)$, however, its boundary is not Zariski closed. 
			We claim that the set $\mathcal{I}(q)^\compl$ is not a spectrahedron. 
			By the characterization of Helton and Vinnikov \cite{helton-vinnikov-2007},
			it suffices to show that $\overline{\mathcal{I}(q)}$ 
			is not rigidly convex. That is, if $h$ is a defining polynomial
			of minimal degree for the component $\mathcal{I}(q)^\compl$, then we have to
			show that a generic line $\ell$ through the interior of $\mathcal{I}(q)^\compl$ does not
			meet the variety $V:=\{\x \in \R^2 \, : \, h(\x) = 0\}$ in exactly $\deg(h)$
			many real points, counting multiplicities.
			However, this can be checked immediately. For example, the 
			line $y_1 = 1/3$ intersects the variety $V$ in exactly
			two real points, and any sufficiently small perturbation of the line
			preserves the number of real intersection points. See Figure \ref{fig:improjComplex} (2c).
			
			\smallskip

		This completes the proof of Theorem~\ref{th:complex-classification1}. We now prove Corollary \ref{co:alg-degrees} by showing that 8 is an upper bound.
		
		\medskip
		
		\noindent	
		\textit{Proof of Corollary \ref{co:alg-degrees}}. For the first four classes we have precisely computed the boundaries $\partial\I(p)$ and they are algebraic with irreducible components of degree at most two. It remains to consider the case (2c), more precisely (2c.2), where
			$p=(z_1 - {\rm i} z_2)(z_1 - \alpha z_2)+\gamma$
			for some $\alpha,\gamma\in\C$, $\alpha\notin\R$, and $\alpha\neq {\rm \pm i}$.
		Using Remark \ref{re:quarticRoots}, we show that the degrees of the irreducible components in the Zariski closure of $\partial\I(p)$ do not exceed $8$. This, together with Example \ref{ex:caseB}, completes the proof of (1). We separate the real and the imaginary parts as before.
		{\small
			\[
			p_{\mathrm{re}} =x_{1}^{2}\!+(\!(\alpha_{\mathrm{im}}+1) y_{2}\!)-\alpha_{\mathrm{re}} x_{2}) x_{1}-\alpha_{\mathrm{im}} x_{2}^{2}+(\!(\!\alpha_{\mathrm{im}}+1) y_{1}-2 \alpha_{\mathrm{re}} y_{2}) x_{2}+\alpha_{\mathrm{re}} y_{2} y_{1}+\alpha_{\mathrm{im}} y_{2}^{2}-y_{1}^{2}\!+\gamma_{\mathrm{re}}=0, 
			\]\[
			p_{\mathrm{im}} =\! ((\alpha_{\mathrm{im}}+1) x_{2}+\alpha_{\mathrm{re}} y_{2}-2 y_{1}) x_{1}-\alpha_{\mathrm{re}} x_{2}^{2}+(\alpha_{\mathrm{re}} y_{1}+2 \alpha_{\mathrm{im}} y_{2}) x_{2}+\alpha_{\mathrm{re}} y_{2}^{2}-(\alpha_{\mathrm{im}}+1) y_{1} y_{2}-\gamma_{\mathrm{im}}= 0.	\vspace{2mm}
			\]
		}	
		First we assume $(\alpha_{\mathrm{im}}+1) x_{2}+\alpha_{\mathrm{re}} y_{2}-2 y_{1}\neq0$. Solving $p_{\mathrm{im}} =0$ for $x_1$
		and substituting in $p_{\mathrm{re}} = 0$ returns
		{\small
			\[
			\Big( \alpha_{\mathrm{im}}(\alpha_{\mathrm{re}}^{2}+(\alpha_{\mathrm{im}}+1)^2) \Big) x_2^4
			-\Big((\alpha_{1}^{2}+\alpha_{2}^{2}+6 \alpha_{2}+1) (-\alpha_{1} y_{2}+y_{1} (\alpha_{2}+1)) \Big) x_2^3\vspace{-1mm}
			+\Big((\alpha_{1}^{2}+5 \alpha_{2}^{2}+14 \alpha_{2}+5) y_{1}^{2}
			\]\vspace{-3mm}\[
			-y_{1} \alpha_{1} (\alpha_{1}^{2}+\alpha_{2}^{2}+14 \alpha_{2}+9) y_{2}+(4 \alpha_{1}^{2}+\alpha_{2} (\alpha_{1}^{2}+(\alpha_{2}-1)^2)) y_{2}^{2}  +(k_{2} \alpha_{1}-2 k_{1} -k_{1} \alpha_{2})\alpha_{2}-k_{2} \alpha_{1}-k_{1}\Big) x_{2}^{2}
			\]\vspace{-3mm}
			\[
			+\Big(8(-\alpha_{2}-1) y_{1}^{3}+8 \alpha_{1} (\alpha_{2}+2) y_{1}^{2} y_{2}-(\alpha_{2} (\alpha_{1}^{2}+\alpha_{2}^{2}-\alpha_{2}-1)+9\alpha_{1}^{2}+1) y_{1} y_{2}^{2}+\alpha_{1} (\alpha_{1}^{2}+(\alpha_{2}^{}-1)^{2}) y_{2}^{3}
			\]\vspace{-3mm}\[
			+4 k_{1}(\alpha_{2}+1) y_{1}+((\alpha_{1}^{2}-(\alpha_{2}-1)^{2}) k_{2}-2 k_{1} \alpha_{1} (\alpha_{2}+1)) y_{2}\Big)x_2       +   4 y_{1}^{4}-8 \alpha_{1} y_{1}^{3} y_{2}+(5 \alpha_{1}^{2}+(\alpha_{2}-1)^{2}) y_{1}^{2} y_{2}^{2}
			\]\vspace{-3mm}\[
			-\alpha_{1} (\alpha_{1}^{2}+(\alpha_{2}-1)^{2}) y_{1} y_{2}^{3}-4 k_{1} y_{1}^{2}+4\alpha_{1} k_{1} y_{1} y_{2}-\alpha_{1} (k_{1} \alpha_{1}+\alpha_{2} k_{2}-k_{2}) y_{2}^{2}-k_{2}^{2}.\vspace{2mm}
			\]
		}		
		Since $\alpha\notin\R$, the leading coefficient is non-zero. Therefore, we have a quartic univariate polynomial in $x_2$. The relevant
 polynomials for the decision of whether this polynomial has a real root for $x_2$ are $P,D$ and the discriminant $\Disc$ from Remark~\ref{re:quarticRoots}. 	
		By computing these polynomials, we observe that $\Disc$ decomposes as $Q_1^2\cdot q$, where $Q_1$ is a quadratic polynomial and $q$ is 
		of degree $8$ in $\mathbf{y}$. 
		The total degrees of $P$ and $D$ are $2$ and $4$, respectively.

		Now let us assume $(\alpha_{\mathrm{im}}+1) x_{2}+\alpha_{\mathrm{re}} y_{2}-2 y_{1} = 0$. If $\alpha_{\mathrm{im}}\neq -1$, then substituting $x_2 = \frac{-\alpha_{\mathrm{re}} y_{2}+2 y_{1}}{\alpha_{\mathrm{im}}+1}$ into $p_{\mathrm{im}}=0$ is the quadratic $Q_1$. Otherwise, the substitution $\alpha_{\mathrm{im}}= -1$ and $y_{1}=\frac{\alpha_{\mathrm{re}} y_{2}}{2}$ in $p_{\mathrm{re}}$ and $	p_{\mathrm{im}}$, and setting $s = 	2p_{\mathrm{im}}-\alpha_{\mathrm{re}}p_{\mathrm{re}}$ simplifies the original system to
		\[\begin{matrix}
			p_{\mathrm{re}} &=& \alpha_{\mathrm{re}}^{2} y_{2}^{2}-4 \alpha_{\mathrm{re}} x_{1} x_{2}-8 \alpha_{\mathrm{re}} x_{2} y_{2}+4 x_{1}^{2}+4 x_{2}^{2}-4 y_{2}^{2}+4 \gamma_{\mathrm{re}}&=&0,\\\vspace{-3mm}\\
			s &=& 2 (2 \alpha_{\mathrm{re}}^{2} x_{1}+3 \alpha_{\mathrm{re}}^{2} y_{2}+4 y_{2})x_{2}-(\alpha_{\mathrm{re}}^{3} y_{2}^{2}+4 \alpha_{\mathrm{re}} x_{1}^{2}+4 \gamma_{\mathrm{re}} \alpha_{1}-4 \gamma_{\mathrm{im}})&=&0.
		\end{matrix}
		\]
		
		If the coefficient of $x_2$ in $s$ is non-zero, then solving $s=0$ for $x_2$ and substituting in $p_{\mathrm{re}}=0$ results in a quartic polynomial in $x_1$ with non-zero leading coefficient. In this case, the polynomials Disc, P, and D from Remark \ref{re:quarticRoots} are all univariate in $y_2$. The decomposition of the discriminant in this case consists of the polynomial $q$ after the substitution  $y_{1}=\frac{\alpha_{\mathrm{re}} y_{2}}{2}$ and the square of a quadratic polynomial $Q_2$. The total degrees of $P$ and $D$ are $2$ and $4$, respectively.
		
		Otherwise, solving $2 \alpha_{\mathrm{re}}^{2} x_{1}+3 \alpha_{\mathrm{re}}^{2} y_{2}+4 y_{2}=0$ for $x_1$ and substituting in $s=0$, results in $Q_2$. In all the cases that we have discussed above, the degree of none of the irreducible factors appearing in the polynomials that could possibly form the $\partial\I(p)$ exceeds 8. Example \ref{ex:caseB} shows an example where this bound is reached. This completes the proof of (1). (2) follows from Theorem \ref{th:RealConicChar}.  \hfill $\Box$
		
		\medskip
		
		We have precisely verified  the imaginary projections for all the normal forms in Theorem \ref{th:conic-classification1} except for (2c.2) . In particular, we have shown that if $p$ is not of the class (2c.2), then $\I(p) = \R^2$ if and only if there exist some $\gamma,\alpha\in\C$, and $\alpha\notin\R$ such that $p$ can be transformed to one of the following normal forms.	
		\begin{equation}\label{list:ConicImprojR2}
			\begin{cases}
				(2a): (z_1 - {\rm i}z_2)^2 + \gamma z_2\quad\text{or}\quad (z_1-{\rm i} z_2)^2 + \gamma\\
				(2b):	z_2 (z_1 - \alpha z_2) + \gamma & \text{for}\quad \gamma=0 \,\,\,\text{or}\,\,\, \alpha_{\mathrm{im}}\gamma_{\mathrm{im}} < 0,\\
				(2c.1): z_1^2+z_2^2+\gamma & \text{for}\quad \gamma_{\mathrm{im}}=0 \,\,\,\text{and}\,\,\, \gamma_{\mathrm{re}}\le0.
			\end{cases}
		\end{equation}
		
		An example for a complex conic of class (2c.2) where the imaginary projection is $\R^2$ is $p = z_1^2 - 3{\rm i}z_1z_2-2z_2^2$. The reason is that for any given $(y_1,y_2)\in\R^2$, the polynomial $p$ vanishes on the point $(-y_2+{\rm i} y_1 , y_1+{\rm i} y_2)$. Answering the following question completes the verification of complex conics with a full-space imaginary projection.
		
		\begin{question}
			Let $p\in\C[z_1,z_2]$ be a complex conic of the form $p=(z_1 - {\rm i}z_2) (z_1 - \alpha z_2) + \gamma$ such that $\alpha\notin\R$ and $\alpha\neq\pm {\rm i}$. Under
			which conditions on the coefficients $\gamma,\alpha\in\C$ does $\mathcal{I}(p)$
			coincide with $\R^2$?
		\end{question}

			%%%%%%%%%%%%%%%%%%%%%%%%%%%%%%%%%%
		%%%%%%%%%%%%%   convexity results  %%%%%%%%%%%%
		%%%%%%%%%%%%%%%%%%%%%%%%%%%%%%%%%%
		%%%%%%%%%%%%%%%%%%%%%%%%%%%%%%%%%%
		%\vspace{-6mm}
		\section{convexity results}\label{se:convex}

		For the case of complex plane conics, we have shown in Theorem \ref{OnebddComp} that there can be at most one bounded component in the complement of its imaginary projection. An example of such a conic is $z_1^2+z_2^2+1 = 0$, where the unique bounded component is the unit disc, which in particular is strictly convex. In the following theorem, we show that for any $k>0$, there exists a complex plane curve whose complement of the imaginary projection has exactly $k$ strictly convex bounded components.
		For the case of real coefficients, only the lower bound of $k$ and no
		exactness result is known 
		(see \cite[Theorem 1.3]{joergens-theobald-hyperbolicity}).

		Allowing non-real coefficients lets us break the symmetry of the imaginary projection with respect to the origin and this enables us to fix the number of components exactly instead of giving a lower bound. 	Furthermore, using a non-real conic which has four strictly convex unbounded components, illustrated in Figure \ref{fig:dist-real}, notably drops the degree of the corresponding polynomial.

		\begin{theorem}\label{th:StrictlyConvexComplex}
			For any $k>0$ there exists a polynomial $p\in\C[z_1,z_2]$ of degree $2\lceil \frac{k}{4}\rceil+2$ such that $\mathcal{I}(p)^\compl$ consists of exactly $k$ strictly convex bounded components.
		\end{theorem}
		
		\begin{proof}
			
			Let $R^{\varphi}$ be the rotation map and $g:\C^2\rightarrow\C^2$ be defined as						
			\[
			g(z_1,z_2) = z_1z_2+2{\rm i}.			
			\]
			Note that the equation
			\begin{equation}\label{eq:m=2}
			\prod_{j=0}^{m-1}(g\circ R^{\pi j/2m})(z_1,z_2) =0 
			\end{equation}				
			where $m=\lceil \frac{k}{4}\rceil$ as before, has $4m$ unbounded components in the complement of its imaginary projection. We need to find a circle that intersects with $k$ of them and does not intersect with the rest $4m-k$ components. By symmetry of the construction of the equation above, the smallest distance between the origin $O$ and each component is the same for all the components. The following picture shows the case $m=2$.  
			\begin{figure}[!htbp]
				\includegraphics[width=6cm]{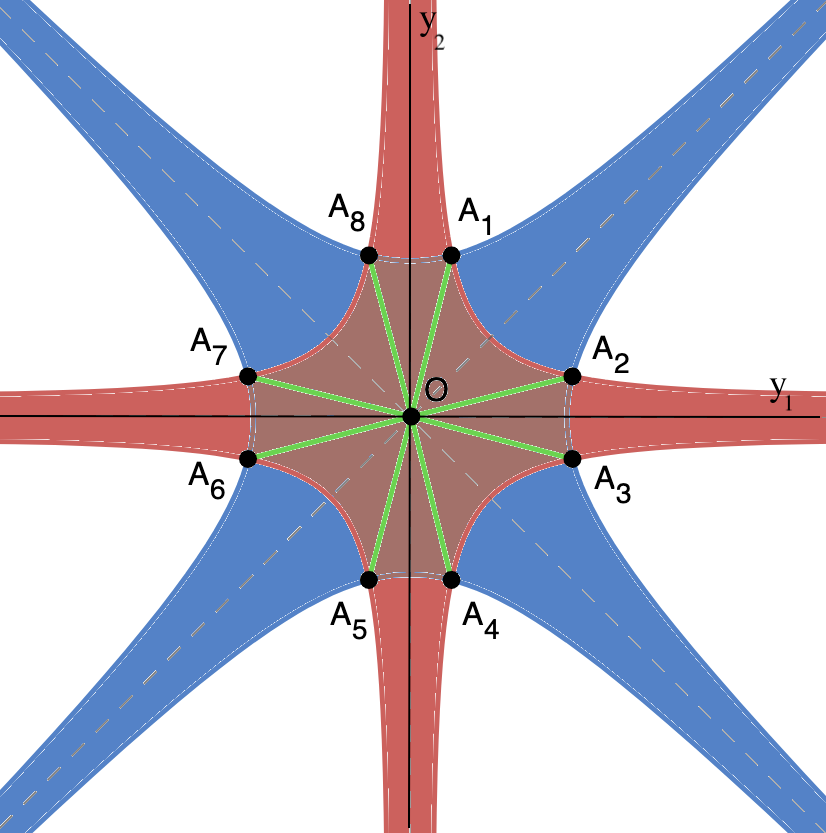}
				\caption{The imaginary projection of (\ref{eq:m=2}) for $m=2$ is the union of the imaginary projections for polynomials corresponding to $j=0$ and $j=1$.}
			\end{figure}	
			
			Let $C$  be the boundary of the imaginary projection of 
			$
			z_1^2 +z_2^2 +r^2
			$
			where $r = |OA_1|$. The center of $C$ is the origin and it passes through all $4m$ points $A_1,\dots,A_{4m}$ that minimize the distance from the origin to each component. A sufficiently small perturbation of the center and the diameter can result in a circle $C'$ with center $(a,b)$ and radius $s$ that only intersects the interiors of the first $k$ unbounded components.
			Now define \[q:=(z_1-{\rm i}a)^2+(z_2-{\rm i}b)^2+s^2.\]
			By Lemma \ref{le:group-actions-improj} and the fact that the imaginary projection of the multiplication of two polynomials is the union of their imaginary projections,
			the polynomial 
			\[
			p := q \cdot \prod_{j=0}^{m-1}(g\circ R^{\pi j/2m})(z_1,z_2),	
			\]		
			has exactly $k$ strictly convex bounded components in $\mathcal{I}(p)^\compl$.
		\end{proof}

		Although, by generalizing from real to complex coefficients, we improved the degree of the desired polynomial from $d=4\lceil \frac{k}{4}\rceil+2$ to $d/2+1$,  it is not the optimal degree.
		For instance if $k=1$, the polynomial $z_1^2+z_2^2+1$ has the desired imaginary projection, while the degree is $2<4$. Thus, we can ask the following question.
		
		\begin{question}\label{ques:deg}
			For $k>0$, what is the smallest integer $d>0$ for which there exists a polynomial $p\in\C[z_1,z_2]$ of degree $d$ such that $\mathcal{I}(p)^\compl$ consists of exactly $k$ strictly convex bounded components.
		\end{question}
		
		\vspace{-5mm}
		
			%%%%%%%%%%%%%%%%%%%%%%%%%%%%%%%%%%
		%%%%%%%%%%%%%    conclusion  %%%%%%%%%%%%
		%%%%%%%%%%%%%%%%%%%%%%%%%%%%%%%%%%
		%%%%%%%%%%%%%%%%%%%%%%%%%%%%%%%%%%
		
		\section{Conclusion and open questions\label{se:outlook}}
		
		We have classified the imaginary projections of complex conics and revealed
		some phenomena for polynomials with complex coefficients in higher degrees and dimensions.
		It seems widely open to come up with a classification of the imaginary projections
		of bivariate cubic polynomials, even in the case of real coefficients. In particular, the maximum number of components in the complement of the imaginary projection for both complex and real polynomials of degree $d$ where $d\ge 3$ is currently unknown. We have shown that in degree two they coincide for real and complex conics, however, this may not be the case for cubic polynomials. 
		
\subsection*{Acknowledgment.} We thank the anonymous referees for
their helpful comments.
		
		\bibliography{1-Nov-22-for-arxiv}
		\bibliographystyle{plain}

	\end{document}